\documentclass[12pt,reqno,oneside,dvipsnames]{amsart}

\usepackage[margin=1in]{geometry}

\usepackage{lmodern}
\usepackage{amsmath,amssymb,amsthm,mathtools}
\usepackage[varbb]{newpxmath}

\usepackage{pgfplots}
\pgfplotsset{compat=1.16}
\usetikzlibrary{arrows.meta,positioning,calc,decorations.pathreplacing}
\usepgfplotslibrary{fillbetween}

\usepackage{graphicx}
\usepackage{enumerate}
\usepackage{bm}
\usepackage{bbm}

\usepackage{verbatim}
\usepackage{hyperref,color}
\usepackage[capitalize,nameinlink]{cleveref}
\hypersetup{
	colorlinks=true,
	pdfpagemode=UseNone,
    citecolor=OliveGreen,
    linkcolor=NavyBlue,
    urlcolor=black,
	pdfstartview=FitW
}
\usepackage{appendix}
\crefname{appsec}{Appendix}{Appendices}
\usepackage{tikz}
\usepackage{aliascnt}

\newtheorem{theorem}{Theorem}[section]

\newaliascnt{lemma}{theorem}
\newtheorem{lemma}[lemma]{Lemma}
\aliascntresetthe{lemma}

\newaliascnt{proposition}{theorem}
\newtheorem{proposition}[proposition]{Proposition}
\aliascntresetthe{proposition}

\newaliascnt{corollary}{theorem}

\aliascntresetthe{corollary}

\theoremstyle{definition}
\newaliascnt{definition}{theorem}

\aliascntresetthe{definition}

\theoremstyle{remark}
\newaliascnt{remark}{theorem}
\newtheorem{remark}[remark]{Remark}
\aliascntresetthe{remark}

\crefname{theorem}{Theorem}{Theorems}
\Crefname{theorem}{Theorem}{Theorems}

\crefname{lemma}{Lemma}{Lemmas}
\Crefname{lemma}{Lemma}{Lemmas}

\crefname{proposition}{Proposition}{Propositions}
\Crefname{proposition}{Proposition}{Propositions}

\crefname{corollary}{Corollary}{Corollaries}
\Crefname{corollary}{Corollary}{Corollaries}

\crefname{definition}{Definition}{Definitions}
\Crefname{definition}{Definition}{Definitions}

\crefname{remark}{Remark}{Remarks}
\Crefname{remark}{Remark}{Remarks}

\theoremstyle{plain}

\theoremstyle{definition}

\newtheorem*{assumption*}{Assumption}

\theoremstyle{remark}

\crefname{lemma}{Lemma}{Lemmas}
\crefname{theorem}{Theorem}{Theorems}
\crefname{definition}{Definition}{Definitions}
\crefname{fact}{Fact}{Facts}
\crefname{claim}{Claim}{Claims}
\crefname{proposition}{Proposition}{Propositions}

\newcommand{\dif}{\,\mathrm{d}}

\newcommand{\E}{\mathbb{E}}

\newcommand{\norm}[1]{\left\lVert #1 \right\rVert}

\newcommand{\ip}[2]{\left\langle #1 , #2 \right\rangle}

\newcommand{\sgn}{\mathrm{sgn}}

\makeatletter
\newcommand{\vast}{\bBigg@{4}}
\newcommand{\Vast}{\bBigg@{5}}
\makeatother

\renewcommand{\epsilon}{\varepsilon}

\newcommand{\N}{\mathbb{N}}

\newcommand{\R}{\mathbb{R}}

\newcommand{\ZZ}{\mathbb{Z}}

\newcommand{\beq}{\begin{equation}}
\newcommand{\eeq}{\end{equation}}

\DeclareMathOperator{\supp}{supp}

\usepackage{microtype}
\usepackage{enumitem}
\usepackage{booktabs}

\definecolor{linkblue}{RGB}{24,72,145}
\theoremstyle{definition}
\theoremstyle{remark}

\newcommand{\cP}{\mathcal P}

\newcommand{\conv}{\operatorname{conv}}
\newcommand{\argmin}{\operatorname*{arg\,min}}

\newcommand{\abs}[1]{\left\lvert #1\right\rvert}

\newcommand{\KL}{D_{\mathrm{KL}}}

\newcommand{\diam}{\operatorname{diam}}

\usepackage{tabularx}
\usepackage{array}

\begin{document}

\title[A Bayesian Proof of the Bernoulli Theorem]{A Bayesian Proof of the Bernoulli Theorem}

\author
[Jingbo Liu
and
Ilias Zadik]{Jingbo Liu$^{\ast}$
and
Ilias Zadik$^{\circ}$}

\thanks{\raggedright$^\ast$Department of Statistics, University of Illinois Urbana-Champaign,
Champaign, IL, 61820, USA.
Email: \texttt{jingbol@illinois.edu}}\thanks{\raggedright$^\circ$Department of Statistics and Data Science, Yale University, New Haven, CT, 06511, USA 
Email: \texttt{ilias.zadik@yale.edu}}

\date{\today}


\begin{abstract}
We give a new proof of the Bernoulli theorem, conjectured by Talagrand
and proved in the seminal work of Bednorz and Lata{\l}a.  Our approach
is based on information-theoretic ideas: lower bounds on the supremum of a
Bernoulli process are translated to the fundamental limits of Bayesian
estimation in a Cauchy additive channel.  This leads to a new
information-theoretic functional that characterizes Bernoulli-process
suprema and plays a role analogous to Fernique's majorizing-measure
functional for Gaussian processes.  The same viewpoint yields a
distributional strengthening: for any prescribed law of the index, we
characterize the largest expected value attainable over all couplings
of that index with the Bernoulli process.  This extends to Bernoulli
processes a phenomenon previously understood for Gaussian processes
through the work of Fernique and Talagrand.
\end{abstract}

\maketitle


{\small
\tableofcontents}
\section{Introduction}
\label{sec:introduction}

Let $T\subset\mathbb R^n$ be finite.  The celebrated majorizing
measure theorem (MMT), whose upper bound is due to Fernique and whose
matching lower bound was proved by Talagrand
\cite{Fernique1975,Talagrand1987}, gives a sharp characterization, up
to universal constants, of the Gaussian width of $T$,
\begin{equation}
  W(T):=\mathbb E\sup_{x\in T}\langle G,x\rangle,
  \qquad G\sim N(0,I_n).
  \label{eq:gaussian-width}
\end{equation}
Write $\cP(T)$ for the set of probability measures on $T$, and set
\[
  B_2(t,r):=\{s\in T:\|s-t\|_2\le r\}.
\]
The MMT may then be stated as follows.

\begin{theorem}[Majorizing Measure Theorem]
\label{thm:MMT-main}
There are universal constants $c,C>0$ such that every finite
$T\subset\mathbb R^n$ satisfies
\begin{equation}
  c\,\gamma_2(T,\|\cdot\|_2)
  \le W(T)
  \le C\,\gamma_2(T,\|\cdot\|_2),
  \label{eq:MMT-main}
\end{equation}
where
\begin{equation}
  \gamma_2(T,\|\cdot\|_2)
  :=
  \inf_{\mu\in\cP(T)}
  \sup_{t\in T}
  \int_0^\infty
  \sqrt{\log\frac{1}{\mu(B_2(t,r))}}\,dr.
  \label{eq:maj}
\end{equation}
\end{theorem}

Beyond Gaussian processes, another central problem is to control
suprema of Bernoulli processes, which arise naturally through
symmetrization in empirical-process theory; see, for example,
\cite{Talagrand1994,Talagrand2021}.  For a finite
$T\subset\mathbb R^n$, define the Bernoulli width
\begin{equation}
  b(T):=\mathbb E\sup_{x\in T}\langle\varepsilon,x\rangle,
  \label{eq:bernoulli-width}
\end{equation}
where $\varepsilon_1,\ldots,\varepsilon_n$ are independent centered
Rademacher variables.  Since $\varepsilon$ is a subgaussian vector,
the standard arguments underlying the upper bound in the MMT give
$b(T)\le C\gamma_2(T,\|\cdot\|_2)$.  Moreover, since
$\|\varepsilon\|_\infty=1$ almost surely,
$b(T)\le\sup_{x\in T}\|x\|_1$.
Talagrand's celebrated Bernoulli conjecture \cite{Talagrand1994}
asserted that the optimal interpolation between these two elementary
bounds is sharp for every finite $T$.  The conjecture was resolved
roughly two decades later in the breakthrough work of Bednorz and
Lata{\l}a \cite{BednorzLatala2014}.  Define
\begin{equation}
  \Lambda(T):=
  \inf_{\substack{T_1,T_2\subset\mathbb R^n\ \mathrm{finite}\\
                  T\subseteq T_1+T_2}}
  \left\{
    \sup_{u\in T_1}\|u\|_1
    +\gamma_2(T_2,\|\cdot\|_2)
  \right\}.
  \label{eq:familiar-bernoulli-functional}
\end{equation}
The Bernoulli theorem of Bednorz and Lata{\l}a can then be stated as
follows.

\begin{theorem}[Bernoulli theorem]
\label{thm:bernoulli-main}
There are universal constants $c,C>0$ such that every finite
$T\subset\mathbb R^n$ satisfies
\begin{equation}
  c\,\Lambda(T)\le b(T)\le C\,\Lambda(T).
  \label{eq:bernoulli-main}
\end{equation}
\end{theorem}

The upper bound follows immediately by combining the two elementary
bounds above.  The central challenge, and hence the principal focus of
\cite{BednorzLatala2014}, is to prove the lower bound.  Because the
proof of Bednorz and Lata{\l}a is technically involved, there has been
considerable interest in finding a more direct and intuitive argument;
see, for example, \cite{TalagrandAbelLecture2024}.

The situation for the Gaussian MMT is now markedly different.
Although Talagrand's original proof was also considered opaque, several
comparatively simple proofs are now available, including geometric
arguments and approaches based on coding theory, interpolation, type
lifting, functional analysis, and Bayesian estimation
\cite{Talagrand1992,Talagrand1996,van2018chaining,Borst2021,Liu2025,
ChuRaginskyISIT2023,Zadik2026BayesianMMT}.  Most relevant here are two
recent developments of the authors \cite{Liu2025,Zadik2026BayesianMMT}.
The first author showed in \cite{Liu2025} that the majorizing-measure
functional $\gamma_2(T,\|\cdot\|_2)$ admits an equivalent
information-theoretic formulation in terms of a rate--distortion
integral, obtained naturally by a lifting argument.  The second author
subsequently showed in
\cite{Zadik2026BayesianMMT} that this formulation yields an independent
information-theoretic proof of the majorizing measure theorem.  That
proof expresses the Gaussian width through a Bayesian estimation
problem for a Gaussian additive channel and then applies
information-theoretic inequalities to obtain the lower bound.

In the present work, we show that the Bayesian approach also tightly
characterizes the Bernoulli width.  The central new observation is that
the appropriate analogue of the Gaussian additive channel is a
\emph{Cauchy additive channel}.  We adapt both the
information-theoretic approach of \cite{Liu2025} and the Bayesian
estimation perspective of \cite{Zadik2026BayesianMMT} to the Bernoulli
width and the Cauchy channel.  This yields a new and comparatively
direct proof of the Bernoulli theorem.

Once the necessary ingredients have been introduced, the proof closely
parallels the Bayesian proof of the Gaussian MMT in
\cite{Zadik2026BayesianMMT}, together with an elementary decomposition
argument.  By contrast, the Bednorz--Lata{\l}a proof uses ingredients
that do not arise in the Gaussian setting, including Lata{\l}a's
principle for dropping coordinates, chaining functionals defined
through chopping maps, and an adaptive decomposition procedure.  Our
approach bypasses this machinery.  One additional ingredient is needed for our proof:
the Cauchy-channel information--estimation inequality proved here,
which plays the role of the Gaussian I--MMSE identity of
\cite{GuoShamaiVerdu}.

\subsection{The distributional Bernoulli theorem}

As the functional $\Lambda(T)$ in Theorem \ref{thm:bernoulli-main} yields a trivial upper bound on $b(T)$, the classical formulation of the Bernoulli theorem is conceptually appealing. At the heart of our approach, however, lies a new formulation of the Bernoulli theorem that characterizes $b(T)$ in terms an information-theoretic functional (eq.\ \eqref{eq:RD-area} below) that is the analogue in the present setting of the Gaussian functional introduced in \cite{Liu2025}. This new formulation not only provides a novel viewpoint on the Bernoulli theorem, but also enables us to establish a significantly stronger result.

Let $\mathrm{Rad}$ denote the uniform law on $\{-1,1\}$, and define, for every finitely supported probability measure $\mu$ on $\mathbb R^n$,
\begin{align}
 B(\mu)
 :=
 \sup_{\substack{X\sim\mu,\ \varepsilon\sim\mathrm{Rad}^{\otimes n}\\
                 \text{jointly coupled}}}
 \mathbb E\langle\varepsilon,X\rangle.
 \label{e_bmu}
\end{align}
That is, $B(\mu)$ is the largest possible expectation of the Bernoulli process evaluated at a random point $X$ with distribution $\mu$. The main result of this paper, Theorem~\ref{thm:distributional-bernoulli-rd} below, yields a distributional strengthening of the Bernoulli theorem: it provides a characterization of $B(\mu)$ for each $\mu$ individually. Since
\[
  b(T) = \sup_{\mu\in\mathcal{P}(T)} B(\mu),
\]
Theorem \ref{thm:bernoulli-main} will follow as a corollary by taking the supremum over $\mu$.

Fixed-law functionals analogous to $B(\mu)$ were introduced in the
Gaussian setting by Fernique, who established several bounds for them
\cite[Theorem~3.1]{Fernique1981}. Talagrand subsequently obtained 
universal-constant characterizations of the Gaussian
fixed-law functional \cite[Theorem~30]{Talagrand1987}; see also
\cite{van2025subgaussian}. By contrast, the methods of
Bednorz--Lata{\l}a do not appear to yield a corresponding fixed-law
version of the Bernoulli theorem. Thus, to our knowledge, our
distributional Bernoulli theorem is the first such fixed-law
characterization for a non-Gaussian process.

\subsection{Organization of the paper}

Section \ref{sec:background} introduces the information-theoretic
objects and states the main distributional theorem.  Sections
\ref{sec:first-area} and \ref{sec:rate-distortion} compare the
Bernoulli width with the integrated Cauchy Bayes risk and then with the
rate--distortion functional.  Section
\ref{sec:direct-rd-bernoulli} converts this functional into the
$\ell_1$-plus-Gaussian decomposition, and Section
\ref{sec:distributional-completion} proves the prescribed-law
distributional result.  The Cauchy information--estimation inequality
is proved in Appendix \ref{sec:i-MMSE}.

\begingroup
\makeatletter
\let\@tocwrite\@gobbletwo
\makeatother

\subsection*{Statement on AI use}

ChatGPT Pro (versions 5.5 and 5.6) was used to generate an initial proof sketch for the information-estimation inequality in Appendix \ref{sec:i-MMSE}, specifically in response to the authors' query regarding whether the inequality holds for the Cauchy additive channel. Beyond this, ChatGPT Pro (versions 5.5 and 5.6) was utilized for accelerating technical calculations (specifically in the proofs of \cref{lem:cauchy-pivotal} and \cref{lem:clipped-rounding}) and language editing throughout the paper. The authors take full responsibility for all content and its accuracy.

\subsection*{Acknowledgments}

We thank Ramon van Handel for sharing unpublished notes, which served as a major motivation for the rate--distortion integral characterization of Bernoulli processes described in \cref{thm:distributional-bernoulli-rd}. In particular, these notes establish a Dudley-type integral bound for sets of sequences with a prescribed empirical distribution, known as exact type classes; together with the Bernoulli theorem as input, this implies the rate--distortion characterization. In the present paper, we instead use the first part of our new Bayesian proof (specifically \cref{eq:combined-main}) to obtain the rate--distortion characterization, which in turn yields an independent proof of the Bernoulli theorem. We are also deeply  grateful to him for several insightful discussions and for his helpful feedback on early versions of this manuscript.

I.Z. is also grateful to Michel Talagrand for encouraging him to pursue a Bayesian proof of Bernoulli's theorem.

J.L.'s research was supported in part by NSF Grant
DMS-2515510.

\subsection*{Note on concurrent work} 
During the preparation of this manuscript, we learned of an independent alternate proof of the Bernoulli theorem generated by an internal model at OpenAI \cite{oai}.
\endgroup

\section{Main result and proof structure}
\label{sec:background}

\subsection{Notation}
\label{sec:infthbackground}

All logarithms are natural and $0\log0=0$.  For $A,B\ge0$,
$A\lesssim B$ means $A\le CB$ for a universal constant $C>0$,
$A\gtrsim B$ means $B\lesssim A$, and $A\asymp B$ means both.
\par\noindent
For probability measures $P,Q$, set
\nopagebreak[4]
\[
  D_{\mathrm{KL}}(P\|Q)
  :=
  \begin{cases}
    \displaystyle
    \int \log\left(\frac{dP}{dQ}\right)\,dP,
      & P\ll Q,\\[0.8em]
    +\infty,
      & P\not\ll Q.
  \end{cases}
\]
For random variables $X,Y,S$, we define
\[
  I(X;Y)
  :=
  D_{\mathrm{KL}}(P_{X,Y}\|P_X\otimes P_Y)
\]
and
\[
  I(X;Y\mid S)
  :=
  \int
  D_{\mathrm{KL}}
  \bigl(P_{X,Y\mid S=s}\|
        P_{X\mid S=s}\otimes P_{Y\mid S=s}\bigr)
  \,P_S(ds).
\]
These definitions apply equally to discrete and continuous observations.
If $X$ is discrete, then
\[
  H(X):=-\sum_x\mathbb P(X=x)\log\mathbb P(X=x)
\]
and
\[
  H(X\mid Y)
  :=
  \mathbb E\left[
    -\sum_x
    \mathbb P(X=x\mid Y)
    \log\mathbb P(X=x\mid Y)
  \right],
  \qquad
  I(X;Y)=H(X)-H(X\mid Y).
\]
For a probability measure $\mu$ on a finite set $T$, we also write
\[
  H(\mu):=-\sum_{x\in T}\mu(x)\log\mu(x).
\]
For $v:T\to\mathbb R^n$, write $v_{\#}\mu$ for the pushforward of
$\mu$ under $v$.
We use without further comment nonnegativity, the chain rule
\[
  I(X;Y,Z)=I(X;Y)+I(X;Z\mid Y),
\]
and the data-processing inequality: if $X\to Y\to Z$ is a Markov
chain, then $I(X;Z)\le I(X;Y)$.

We will also use posterior replicas. Given two random variables $X,Y$ where $Y$ is a (random) function of $X$ (in statistical terminology, $Y$ is often called an observation of the signal $X$), a
posterior replica $X^{(1)}$ is a random variable with conditional law
$P_{X\mid Y}$.  Two posterior replicas $X^{(1)},X^{(2)}$ are sampled
conditionally independently of each other and of $X$ given $Y$.
Thus $X,X^{(1)},X^{(2)}$ are conditionally i.i.d.\ given $Y$, and in
particular
\[
  \mathbb E f(X,X^{(1)})
  =
  \mathbb E f(X^{(1)},X^{(2)})
\]
for every integrable measurable $f$.  We refer to this elementary
posterior-sampling identity as the Nishimori identity.

\subsection{Key objects}

For $t>0$ and $x,x'\in\mathbb R^n$, define the capped quadratic
distortion
\begin{equation}
  \varphi_t(x,x')
  :=
  \sum_{i=1}^n
  \left(
    1\wedge\frac{|x_i-x_i'|^2}{t^2}
  \right).
  \label{eq:capped-distortion}
\end{equation}
This distortion already appears in Talagrand's geometric formulation
of the Bernoulli problem \cite[Eq.~(4.4)]{Talagrand1994}.  It captures
the scale-dependent quadratic geometry that contributes to the
Gaussian component of the Bernoulli decomposition.  It also motivates
our use of the Cauchy channel: the Cauchy information--estimation
inequality proved below naturally produces precisely the capped loss
$\varphi_t$.

Fix a finite $T\subset\mathbb R^n$.  For $1\le p\le\infty$, write
\[
  \diam_p(T):=\max_{x,u\in T}\|x-u\|_p.
\]
For $\mu\in\cP(T)$, define
\begin{equation}
  \mathrm{RD}_\mu(t)
  :=
  \inf_{\substack{X\sim\mu,\ \widehat X\sim\mu}}
  \left\{
    I(X;\widehat X)+\mathbb E\varphi_t(X,\widehat X)
  \right\},
  \label{eq:rate-distortion}
\end{equation}
where the infimum is over all couplings with both marginals equal to
$\mu$.  Thus $\mathrm{RD}_\mu(t)$ is a self-coupled Lagrangian
rate--distortion value.  The set-level rate--distortion functional is
\begin{equation}
  R(T)
  :=
  \sup_{\mu\in\cP(T)}
  \int_0^\infty \mathrm{RD}_\mu(t)\,dt.
  \label{eq:RD-area}
\end{equation}

We next define the Cauchy channel and its Bayes risk.  Let
$Z=(Z_1,\ldots,Z_n)$ have independent standard Cauchy coordinates,
each with density
\[
  z\longmapsto\frac{1}{\pi(1+z^2)}.
\]
For $\mu\in\cP(T)$, let $X\sim\mu$ be independent of $Z$, and observe
\begin{equation}
  Y_t:=X+tZ.
  \label{eq:product-cauchy-channel}
\end{equation}
The corresponding capped minimum mean-square error is
\begin{equation}
  \mathrm{cMMSE}_\mu(t)
  :=
  \inf_{\widehat x:\mathbb R^n\to\mathbb R^n}
  \mathbb E\varphi_t\bigl(X,\widehat x(Y_t)\bigr),
  \label{eq:cauchy-cmmse}
\end{equation}
where the infimum is over measurable estimators, which need not be
$T$-valued.

For every finitely supported probability measure $\mu$ on
$\mathbb R^n$, define the prescribed-law Gaussian functional
\[
  G(\mu)
  :=
  \sup_{\substack{Y\sim\mu,\ G_0\sim N(0,I_n)\\
                  \text{jointly coupled}}}
  \mathbb E\langle Y,G_0\rangle.
\]
Finally, with $\operatorname{id}_T(x)=x$, set
\begin{equation}
  \delta_T(\mu)
  :=
  \inf_{a:T\to\mathbb R^n}
  \left\{
    \mathbb E_{X\sim\mu}\|a(X)\|_1
    +G\bigl((\operatorname{id}_T-a)_{\#}\mu\bigr)
  \right\}.
  \label{eq:distributional-decomposition}
\end{equation}

\subsection{Main result}
\label{sec:mainres}

\begin{theorem}[Distributional Bernoulli rate--distortion theorem]
\label{thm:distributional-bernoulli-rd}
\label{thm:bernitintro}
\label{thm:bernmeasintro}
For every finite $T\subset\mathbb R^n$ and every
$\mu\in\cP(T)$,
\begin{align}
  B(\mu)
  \asymp
  \int_0^\infty \mathrm{RD}_\mu(t)\,dt
  \asymp
  \delta_T(\mu),
  \label{e14}
\end{align}
where the implicit constants are universal.  Consequently,
\begin{align}
  b(T)
  \asymp
  R(T)
  \asymp
  \Lambda(T).
  \label{e15}
\end{align}
\end{theorem}

The prescribed-law statement \eqref{e14} refines the classical Bernoulli
theorem by characterizing the largest expected Bernoulli-process value
attainable over all couplings for which the index has a prescribed law.
An analogous fixed-law characterization is available for Gaussian
processes \cite[Theorem~30]{Talagrand1987}, whereas existing proofs of the
Bernoulli theorem do not appear to yield such a prescribed-law result.
Furthermore, the proof of \eqref{e14} reveals that the decomposition in $\delta_T$ (function $a$ in \eqref{eq:distributional-decomposition}) can be constructed using information-constrained near-optimal transport couplings.
Supremizing $\mu$ over $T$ in \eqref{e14} gives \eqref{e15}.

\subsection{Proof structure and ideas}
The proof of Theorem~\ref{thm:distributional-bernoulli-rd} closes a
cycle of three comparisons.  Type lifting \cite{Liu2025}, which replaces
an index law by the uniform law on sequences with that empirical
distribution together with a comparison with the fundamental Bayesian estimation limits of a Cauchy channel, gives
$\int_0^\infty\mathrm{RD}_\mu(t)\,dt\lesssim B(\mu)$, 
an elementary
coupling gives $B(\mu)\lesssim\delta_T(\mu)$, and a
multiscale distributional decomposition gives
$\delta_T(\mu)\lesssim\int_0^\infty\mathrm{RD}_\mu(t)\,dt$.  Taking
the supremum over the $\mu$ and applying the corresponding minimax
decomposition yields the set-level assertion.

Here is how the three comparisons above are distributed through the
paper.  Sections~\ref{sec:first-area} and \ref{sec:rate-distortion}
first prove \cref{thm:cauchy-areas}, and hence
$\int_0^\infty\mathrm{RD}_\mu(t)\,dt\lesssim b(T)$ for every $\mu$.
Section~\ref{sec:direct-rd-bernoulli} proves
$\delta_T(\mu)\lesssim\int_0^\infty\mathrm{RD}_\mu(t)\,dt$ and uses a
minimax argument to prove \cref{thm:rd-bernoulli}.  Finally,
Section~\ref{sec:distributional-completion} applies the set-level
estimate to uniform laws on empirical-distribution fibers to obtain
$\int_0^\infty\mathrm{RD}_\mu(t)\,dt\lesssim B(\mu)$, proves
$B(\mu)\lesssim\delta_T(\mu)$ by Gaussianization, and completes the
prescribed-law part of
\cref{thm:distributional-bernoulli-rd}.

At the set level, \cref{thm:cauchy-areas} gives
$R(T)\lesssim b(T)$, \cref{thm:rd-bernoulli} gives
$\Lambda(T)\lesssim R(T)$, and the elementary $\ell_1$ and Gaussian
bounds applied to a decomposition $T\subseteq T_1+T_2$ give
$b(T)\lesssim\Lambda(T)$.  The two non-elementary set-level estimates
are stated next.

\begin{theorem}[Bernoulli width dominates the rate--distortion area]
\label{thm:cauchy-areas}
For every finite $T\subset\mathbb R^n$ and every 
$\mu\in\cP(T)$,
\begin{align}
  b(T)
  &\ge
  \frac{1}{4\pi}
  \int_0^\infty \mathrm{cMMSE}_\mu(t)\,dt,
  \label{eq:first-area-main}\\
  \int_0^\infty \mathrm{cMMSE}_\mu(t)\,dt
  &\ge
  0.05\int_0^\infty\mathrm{RD}_\mu(t)\,dt.
  \label{eq:rd-main}
\end{align}
Consequently,
\begin{equation}
  b(T)\ge\frac{1}{80\pi}R(T).
  \label{eq:combined-main}
\end{equation}
\end{theorem}

The Gaussian argument that motivates \cref{thm:cauchy-areas} can be
summarized without assuming prior familiarity with information theory.
In a Gaussian additive channel, the Gaussian width can be expressed as
the area under the estimation-error curve of a least-squares estimator (see also \cite{pathak2026gaussian}).
Replacing that estimator by the Bayes-optimal estimator can only
decrease the error.  The Gaussian mutual-information/minimum
mean-square error (I--MMSE) identity \cite{GuoShamaiVerdu} then relates
this Bayes risk to the decay of mutual information, which in turn
controls a rate--distortion quantity.

For Bernoulli processes, an interpolation potential based on the least
absolute deviation estimator in the Cauchy channel yields the first
inequality in \cref{thm:cauchy-areas}.  Its endpoint identity is
\[
  \lim_{t\to\infty}
  \mathbb E\bigl[
    \|Y_t-X\|_1
    -\|Y_t-\Pi_K(Y_t)\|_1
  \bigr]
  =
  b(T),
  \qquad
  \Pi_K(y)\in
  \argmin_{u\in\operatorname{conv}(T)}\|y-u\|_1.
\]
Differentiating the potential produces a capped estimation risk.  The
Cauchy information--estimation inequality proved in Appendix
\ref{sec:i-MMSE}, together with a posterior-replica coupling, gives the
second inequality.  Motivated by the question raised in
\cite{Verdu2023}, we derive this Cauchy-channel relation for finite
input distributions and obtain the required bound in terms of integrated capped
Bayes risk.

\begin{table}[t]
\centering
\small
\renewcommand{\arraystretch}{1.3}
\begin{tabularx}{\linewidth}{
  @{}
  >{\raggedright\arraybackslash}p{0.19\linewidth}
  >{\raggedright\arraybackslash}X
  >{\raggedright\arraybackslash}X
  @{}
}
\toprule
&
\textbf{Gaussian MMT}
&
\textbf{Bernoulli theorem}
\\
\midrule
Additive channel
&
$Y_t^{\rm G}=X+tG$, $G\sim N(0,I_n)$
&
$Y_t^{\rm C}=X+tZ$, $Z_i$ i.i.d.\ standard Cauchy
\\[0.4em]
Natural loss
&
$\|x-\widehat x\|_2^2$
&
$\varphi_t(x,\widehat x)$
\\[0.4em]
Width-to-risk step
&
Gaussian width controls integrated Gaussian MMSE
\cite{Zadik2026BayesianMMT}
&
Bernoulli width controls integrated Cauchy cMMSE
(\cref{sec:first-area})
\\[0.4em]
Information--estimation step
&
Gaussian I--MMSE identity \cite{GuoShamaiVerdu}
&
Cauchy information--estimation inequality
(Appendix \ref{sec:i-MMSE})
\\[0.4em]
Risk-to-rate--distortion step
&
Integrated MMSE controls a rate--distortion functional
\cite{Zadik2026BayesianMMT}
&
Integrated cMMSE controls the rate--distortion functional
(\cref{sec:rate-distortion})
\\[0.4em]
Functional step
&
The rate--distortion dual controls
$\gamma_2(T,\|\cdot\|_2)$ \cite{Liu2025}
&
$R(T)$ controls $\Lambda(T)$
(\cref{sec:direct-rd-bernoulli})
\\
\bottomrule
\end{tabularx}
\caption{Dictionary between the Bayesian proofs in the Gaussian and
Bernoulli settings.}
\label{tab:bayesian-dictionary}
\end{table}

\begin{theorem}[Rate--distortion area dominates the Bernoulli functional]
\label{thm:rd-bernoulli}
There is a universal constant $c>0$ such that every finite
$T\subset\mathbb R^n$ satisfies
\begin{equation}
  R(T)\ge c\,\Lambda(T).
  \label{eq:rd-bernoulli}
\end{equation}
\end{theorem}

For a prescribed-law $\mu$, we prove that the rate--distortion integral controls
$\delta_T(\mu)$.  A minimax argument then converts
$\sup_{\mu}\delta_T(\mu)$ into a single decomposition of $T$ with
an $\ell_1$-bounded part and a part of controlled Gaussian width.
The Gaussian MMT turns the latter bound into
$\gamma_2$-control, proving \cref{thm:rd-bernoulli}.

\section{Bernoulli width and the Cauchy cMMSE area bound}
\label{sec:first-area}

The goal of this section is to prove \eqref{eq:first-area-main}, the
first comparison in \cref{thm:cauchy-areas}.  We construct an
interpolation potential whose total variation is $b(T)$, show that its
derivative controls the capped error of a leave-one-coordinate-out
estimator, and then average over the prior and optimize over estimators.

\subsection{Area theorem for Bernoulli width and absolute-deviation estimation}\label{sec:first-area_0}

We first construct the interpolation potential and identify its
endpoints, thereby expressing $b(T)$ as the integral of its derivative.

Consider the convex hull $K:=\conv(T)$ and define the
$\ell_1$-distance to $K$ by
\begin{equation}
  d_K(y):=\inf_{u\in K}\norm{y-u}_1.
  \label{eq:l1-distance}
\end{equation}

\begin{remark}
It is useful to adopt an information-theoretic or statistical viewpoint.
Fix a Borel measurable selector
\[
  \Pi_K(y)\in\argmin_{u\in K}\|y-u\|_1,
  \qquad y\in\mathbb R^n.
\]
Such a selector exists because $K$ is compact and the objective is
jointly continuous.  In the Cauchy channel
\eqref{eq:product-cauchy-channel}, $\Pi_K(Y_t)$ is a least absolute
deviation estimator and
\[
  d_K(y)=\|\Pi_K(y)-y\|_1.
\]
\end{remark}

Now, fix a ``planted'' point $x\in T$.  For $t\ge0$, define the
interpolation potential
\begin{equation}
  \Phi_x(t)
  :=
  \E_Z\left[
    t\norm{Z}_1-d_K(x+tZ)
  \right].
  \label{eq:first-area-potential-formal}
\end{equation}
Equivalently,
\begin{equation}
  \Phi_x(t)
  =
  \E_Z\sup_{u\in K}
  \left\{
    t\norm{Z}_1-\norm{x+tZ-u}_1
  \right\}.
  \label{eq:first-area-potential}
\end{equation}

\begin{remark}
Notice that although $\E\norm Z_1=\infty$, $\Phi_x(t)$ is well-defined.
Indeed, for every $z\in\mathbb R^n$, $t\ge0$, and $x\in T$, choosing
$u=x$ in \cref{eq:first-area-potential} for the lower bound and using
the triangle inequality for the upper bound give
\begin{equation}
  0
  \le t\norm z_1-d_K(x+tz)
  \le
  \sup_{u\in K}\norm{u-x}_1.
  \label{eq:potential-bound}
\end{equation}
\end{remark}

The first key lemma connects the Bernoulli width $b(T)$ with
$\Phi_x(t)$.

\begin{lemma}[Cauchy-area theorem for Bernoulli width]
\label{lem:first-area-endpoints}
For any $x \in T$, the function $\Phi_x: [0,\infty) \rightarrow [0,+\infty)$ is nondecreasing and concave on $[0,\infty)$ with $\Phi_x(0)=0$. Moreover, it satisfies
\begin{equation}
  b(T)=\lim_{t \to \infty} \Phi_x(t)=\int_0^\infty \Phi_x'(t)\,dt.
  \label{eq:first-area-integral}
\end{equation}
The derivative is understood almost everywhere.
\end{lemma}

\begin{proof}
For $z\in\mathbb R^n$, set
\[
  F_z(t):=t\|z\|_1-d_K(x+tz),
  \qquad t\ge0.
\]
Clearly, $F_z(0)=-d_K(x)=0$.  Since $d_K$ is convex, $F_z$ is
concave.  Moreover, if $0\le t_2\le t_1$, the $1$-Lipschitz property
of $d_K$ with respect to $\|\cdot\|_1$ gives
\[
  d_K(x+t_1z)-d_K(x+t_2z)
  \le (t_1-t_2)\|z\|_1,
\]
and hence $F_z(t_1)\ge F_z(t_2)$.  Taking expectations shows that
$\Phi_x$ is concave and nondecreasing.

Let
\[
  D_x:=\sup_{u\in K}\|u-x\|_1<\infty.
\]
By \cref{eq:potential-bound}, $0\le F_z(t)\le D_x$ for every $z$ and
$t\ge0$.  For every fixed $z$, $F_z(t)\to0$ as $t\downarrow0$.
Dominated convergence therefore gives
$\Phi_x(t)\to0=\Phi_x(0)$ as $t\downarrow0$.

Now fix $z$ such that $z_i\ne0$ for every $i$.  Compactness of $K$
implies, for every $i$,
\[
  \lim_{t\to\infty}
  \sup_{u\in K}
  \left|
    |x_i+tz_i-u_i|-t|z_i|
    -\sgn(z_i)(x_i-u_i)
  \right|
  =0.
\]
Indeed, when $z_i>0$, the expression vanishes identically for all
sufficiently large $t$, uniformly in $u\in K$, and the case $z_i<0$
is identical.  Summing over the coordinates and using uniformity in
$u$ gives
\[
  \lim_{t\to\infty}F_z(t)
  =
  \sup_{u\in K}\langle\sgn(z),u-x\rangle.
\]
Since $Z$ has independent standard Cauchy coordinates,
$\sgn(Z)$ is a Rademacher vector.  Dominated convergence, with the
dominating constant $D_x$, yields
\begin{align*}
  \lim_{t\to\infty}\Phi_x(t)
  &=
  \mathbb E\sup_{u\in K}\langle\varepsilon,u-x\rangle \\
  &=
  \mathbb E\sup_{u\in T}\langle\varepsilon,u\rangle
  -\mathbb E\langle\varepsilon,x\rangle \\
  &=b(T).
\end{align*}
Here we used that a linear functional has the same supremum over
$T$ and over $\conv(T)$, and that $\mathbb E\varepsilon=0$.

As a finite concave function, $\Phi_x$ is absolutely continuous on
every compact subinterval of $(0,\infty)$.  Thus, for
$0<a<b<\infty$,
\[
  \Phi_x(b)-\Phi_x(a)
  =
  \int_a^b\Phi_x'(t)\,dt.
\]
The derivative is nonnegative almost everywhere because $\Phi_x$ is
nondecreasing.  Letting $a\downarrow0$ and $b\uparrow\infty$, and using
the continuity at zero and the endpoint limit proved above, gives
\[
  \int_0^\infty\Phi_x'(t)\,dt
  =
  \lim_{b\to\infty}\Phi_x(b)-\Phi_x(0)
  =
  b(T).
\]
\end{proof}

\begin{remark}
The preceding proof applies whenever $Z$ has i.i.d.\ coordinates
with a symmetric continuous distribution.  The Cauchy assumption
becomes essential in the next subsection.
\end{remark}

\subsection{The potential's derivative dominates the error of a leave-one-coordinate-out estimator}\label{sec:second-area}

We next connect the derivative of $\Phi_x(t)$ with the error of a
leave-one-coordinate-out estimator in the Cauchy channel.

Fix $i\in[n]$ and $w\in\R^{n-1}$.  Define
\begin{equation}
g_{i,w}(u):=d_K(w_1,\ldots,w_{i-1},u,w_i,\ldots,w_{n-1}).
  \label{eq:profile}
\end{equation}
The function $g_{i,w}$ is convex, $1$-Lipschitz, and coercive, so
its set of minimizers is a nonempty compact interval.  Denote its left
endpoint by $\zeta_i(w)$.  This is a Borel measurable tie-breaking
rule: every minimizer lies in the fixed compact interval
\[
  I_i:=\left[\min_{v\in K}v_i,\max_{v\in K}v_i\right],
\]
and the measurable maximum theorem applied to the jointly continuous
map $(w,u)\mapsto g_{i,w}(u)$ shows that the left endpoint of its
argmin over $I_i$ is measurable.

For $y=(y_1,\ldots,y_n)\in\mathbb R^n$, write
\[
  y_{-i}:=(y_1,\ldots,y_{i-1},y_{i+1},\ldots,y_n).
\]
Define
\begin{equation}
  \widehat x^{\mathrm{loc}}(y):=\zeta_i(y_{-i}),
  \qquad
  \widehat x^{\mathrm{loc}}(y)
  :=
  \bigl(
    \widehat x_1^{\mathrm{loc}}(y),\ldots,
    \widehat x_n^{\mathrm{loc}}(y)
  \bigr).
  \label{eq:pivotal-estimator}
\end{equation}
Each coordinate ignores the corresponding coordinate of $y$, and the
measurability of the functions $\zeta_i$ shows that
$\widehat x^{\mathrm{loc}}:\mathbb R^n\to\mathbb R^n$ is a Borel
measurable estimator.  This coordinatewise estimator need not belong
to $K$ and should not be confused with the $K$-valued selector
$\Pi_K$ above.

\begin{remark}
We interpret $\widehat x_i^{\mathrm{loc}}(y)$ as a
leave-one-coordinate-out
$\ell_1$-estimator of $x_i$ from the Cauchy observation
$Y=y=x+tZ$.  It chooses the $i$-th coordinate substitution that
minimizes the $\ell_1$-distance to $K$.  This estimator arises
naturally from the derivative of the potential $\Phi_x(t)$.
\end{remark}
The following key lemma shows that, for almost every $t>0$, the
derivative of $\Phi_x(t)$ controls the expected capped quadratic error
of this estimator, measured by $\varphi_t$ from
\cref{eq:capped-distortion}.
\begin{lemma}
\label{lem:cauchy-pivotal}
For any $x\in T$ and almost every $t>0$,
\begin{equation}
  \Phi_x'(t)
  \ge
  \frac{1}{4\pi}
  \E_Z\varphi_t\bigl(x,\widehat x^{\mathrm{loc}}(x+tZ)\bigr).
  \label{eq:cauchy-pivotal-bound}
\end{equation}
\end{lemma}
\begin{proof}
For fixed $z\in\mathbb R^n$, let
\[
  F_z(t):=t\|z\|_1-d_K(x+tz).
\]
The function $F_z$ is nondecreasing and concave.  It is also locally
absolutely continuous, since
\[
  |F_z(t)-F_z(s)|
  \le 2|t-s|\|z\|_1.
\]
Let $q_z(t):=F_z'(t)$ at points where the derivative exists, and set
$q_z(t):=0$ on the remaining null set.  Then $q_z(t)\ge0$.

The convex function $d_K$ is differentiable Lebesgue almost everywhere
on $\mathbb R^n$.  Moreover, because it is $1$-Lipschitz with respect
to $\|\cdot\|_1$, every gradient at a differentiability point satisfies
$\|\nabla d_K\|_\infty\le1$.  For every $t>0$, the random vector
$x+tZ$ has an absolutely continuous distribution.  Consequently,
Fubini's theorem and the chain rule give, for
$dt\otimes P_Z$-almost every $(t,z)$,
\[
  q_z(t)
  =
  \sum_{i=1}^n
  \left[
    |z_i|-z_i\partial_i d_K(x+tz)
  \right].
\]
Every summand on the right is nonnegative.

For $0<a<b<\infty$, absolute continuity of $F_z$ and Tonelli's theorem
therefore give
\begin{align*}
  \Phi_x(b)-\Phi_x(a)
  &=
  \mathbb E_Z\bigl[F_Z(b)-F_Z(a)\bigr] \\
  &=
  \mathbb E_Z\int_a^b q_Z(t)\,dt \\
  &=
  \int_a^b\mathbb E_Z q_Z(t)\,dt.
\end{align*}
It follows that, for almost every $t>0$,
\begin{equation}
  \Phi_x'(t)
  =
  \mathbb E_Z\sum_{i=1}^n
  \left[
    |Z_i|-Z_i\partial_i d_K(x+tZ)
  \right].
  \label{eq:first-area-derivative}
\end{equation}
This argument does not require $\mathbb E|Z_i|<\infty$.

Fix such a $t$ and a coordinate $i$.  Write
\[
  W:=Y_{t,-i}=x_{-i}+tZ_{-i}.
\]
By Fubini's theorem, for $P_W$-almost every $w$, the identity
\[
  \partial_i d_K(w,x_i+tZ_i)
  =
  g_{i,w}'(x_i+tZ_i)
\]
holds almost surely in $Z_i$.  Notice also that $Z_i$ is independent
of $W$ and remains a standard Cauchy variable after conditioning on
$W=w$.

Fix such a $w$, and write
\[
  \zeta:=\zeta_i(w),
  \qquad
  r:=\frac{|x_i-\zeta|}{t}.
\]
The function $g_{i,w}$ is convex and $1$-Lipschitz, so
$|g_{i,w}'|\le1$ wherever its derivative exists.  In particular,
\[
  |z|-z\,g_{i,w}'(x_i+tz)\ge0
\]
for almost every $z\in\mathbb R$.

Suppose first that $x_i\le\zeta$.  For $0\le z\le r$, one has
$x_i+tz\le\zeta$.  Since $\zeta$ is the left endpoint of the minimizer
interval of $g_{i,w}$, convexity gives
\[
  g_{i,w}'(x_i+tz)\le0
\]
for almost every $z\in[0,r]$.  Hence
\[
  |z|-z\,g_{i,w}'(x_i+tz)\ge z
\]
on this interval, almost everywhere.  Using the standard Cauchy
density and discarding the nonnegative contribution from its
complement, we obtain
\begin{align}
  &\mathbb E_{Z_i}
  \left[
    |Z_i|-Z_i g_{i,w}'(x_i+tZ_i)
  \right]
  \notag\\
  &\qquad\ge
  \frac1\pi\int_0^r\frac{z}{1+z^2}\,dz
  =
  \frac{1}{2\pi}\log(1+r^2).
  \label{eq:cauchy-log-bound}
\end{align}
If $x_i\ge\zeta$, the same argument applies on $[-r,0]$, where a
derivative of $g_{i,w}$ is nonnegative.

Since $\widehat x_i^{\mathrm{loc}}(Y_t)=\zeta_i(W)$, we have shown
that, for $P_W$-almost every $w$,
\begin{align}
  &\mathbb E_{Z_i}
  \left[
    |Z_i|-Z_i\partial_i d_K(x+tZ)
    \,\middle|\,W=w
  \right]
  \notag\\
  &\qquad\ge
  \frac{1}{2\pi}
  \log\left(
    1+
    \frac{|x_i-\widehat x_i^{\mathrm{loc}}(Y_t)|^2}{t^2}
  \right).
  \label{eq:cauchy-log-bound_2}
\end{align}
For every $r\ge0$,
\begin{equation}
  \frac{1}{2\pi}\log(1+r^2)
  \ge
  \frac{1}{4\pi}(1\wedge r^2).
  \label{eq:log-dominates-cap}
\end{equation}
Indeed, $\log(1+r^2)\ge r^2/2$ when $r\le1$, while
$\log(1+r^2)\ge\log2>1/2$ when $r\ge1$.  Averaging over $W$, summing
over $i$, and using \cref{eq:first-area-derivative} now yields
\[
  \Phi_x'(t)
  \ge
  \frac{1}{4\pi}
  \mathbb E_Z
  \varphi_t\bigl(x,\widehat x^{\mathrm{loc}}(x+tZ)\bigr).
\]
\end{proof}

\subsection{The Bernoulli width upper bounds the cMMSE integral}
The preceding bounds hold for every fixed planted point $x\in T$.
Now let $X\sim\mu$, where $\mu\in\cP(T)$, and regard $\mu$ as the
prior in the Cauchy channel.

For completeness, the function
$t\mapsto\mathrm{cMMSE}_\mu(t)$ is Borel measurable on $(0,\infty)$.
Indeed, let
\[
  J_T:=
  \prod_{i=1}^n
  \left[\min_{x\in T}x_i,\max_{x\in T}x_i\right].
\]
Coordinatewise projection of an action onto $J_T$ cannot increase
$\varphi_t(x,\cdot)$ for any $x\in T$, so the conditional Bayes-action
optimization may be restricted to the compact set $J_T$.  Since $T$
is finite, the Cauchy posterior weights are Borel functions of
$(t,y)$, and the conditional risk is continuous in the action.  The
measurable minimum theorem therefore shows that the conditional Bayes
risk, and hence $\mathrm{cMMSE}_\mu(t)$, is Borel measurable.

Optimizing over all estimators instead of using the
leave-one-coordinate-out estimator in \cref{lem:cauchy-pivotal} yields
\cref{eq:first-area-main}.
\begin{proposition}[cMMSE-area inequality]
\label{prop:first-area}
For every $\mu\in\cP(T)$,
\begin{equation}
  b(T)
  \ge
  \frac{1}{4\pi}
  \int_0^\infty \mathrm{cMMSE}_\mu(t)\,dt.
  \label{eq:first-area-proposition}
\end{equation}
\end{proposition}

\begin{proof}
For each $x\in T$, \cref{lem:cauchy-pivotal} holds outside a Lebesgue
null set $N_x\subset(0,\infty)$.  Since $T$ is finite,
\[
  N:=\bigcup_{x\in T}N_x
\]
is again a null set.  Thus, for every $t\notin N$, we may average the
pivotal inequality simultaneously over $X\sim\mu$ to obtain
\[
  \mathbb E_X\Phi_X'(t)
  \ge
  \frac{1}{4\pi}
  \mathbb E\varphi_t\bigl(X,\widehat x^{\mathrm{loc}}(Y_t)\bigr)
  \ge
  \frac{1}{4\pi}\mathrm{cMMSE}_\mu(t).
\]
The second inequality holds because $\widehat x^{\mathrm{loc}}$ is a
Borel measurable estimator and is therefore feasible in
\cref{eq:cauchy-cmmse}.

Finally, $\Phi_x'(t)\ge0$ almost everywhere.  Tonelli's theorem and
\cref{lem:first-area-endpoints} give
\begin{align*}
  \int_0^\infty\mathbb E_X\Phi_X'(t)\,dt
  &=
  \mathbb E_X\int_0^\infty\Phi_X'(t)\,dt \\
  &=
  \mathbb E_X b(T)
  =
  b(T).
\end{align*}
Consequently,
\[
  b(T)
  \ge
  \frac{1}{4\pi}
  \int_0^\infty\mathrm{cMMSE}_\mu(t)\,dt.
\]
\end{proof}

\section{From the Cauchy cMMSE to rate--distortion}
\label{sec:rate-distortion}

The goal of this section is to prove \eqref{eq:rd-main}, the second
comparison in \cref{thm:cauchy-areas}.  Together with
\cref{sec:first-area}, this completes the proof of that theorem.  The
key input is a Cauchy information--estimation inequality, proved in
Appendix~\ref{sec:i-MMSE}, that plays the role of the I--MMSE identity
for the Gaussian additive channel.  A posterior-replica coupling then
converts this information bound into the required rate--distortion
estimate.

For $t\ge0$, consider the Cauchy observation $Y_t=X+tZ$, where
$X\sim\mu$ is independent of $Z$, and let
\begin{equation}
  I_\mu(t):=I(X;Y_t).
\end{equation}

\begin{proposition}\label{prop:multii_I-MMSE_0}
For every $\mu\in\cP(T)$ and $t>0$,
\begin{align}
  I_\mu(t) \le
16\int_t^\infty\frac{\mathrm{cMMSE}_\mu(s)}{s}\,ds.
\end{align}
\end{proposition}
This is the second assertion of \cref{prop:multii_I-MMSE}, proved in
Appendix~\ref{sec:i-MMSE} by a direct calculation using the harmonicity
of the Cauchy density.

For $t>0$, let $X_t^{(1)},X_t^{(2)}$ be conditionally independent
vector-valued draws from the posterior law $P_{X\mid Y_t}$.  Define
\begin{equation}
\begin{aligned}
  \mathrm{rcMMSE}_{\mu}(t)
  &:=
  \mathbb E\varphi_t\bigl(X_t^{(1)},X_t^{(2)}\bigr)
  =
  \sum_{i=1}^n\mathrm{rcMMSE}_{i,\mu}(t),
  \\
  \mathrm{rcMMSE}_{i,\mu}(t)
  &:=
  \mathbb E\left[
    1\wedge
    \frac{|X_{t,i}^{(1)}-X_{t,i}^{(2)}|^2}{t^2}
  \right].
\end{aligned}
  \label{eq:diagonal-coordinate-risk}
\end{equation}
The following lemma compares this posterior-replica risk pointwise
with the capped Bayes risk.
\begin{lemma}
\label{lem:replica-center_1}
For every $t>0$,
$ \mathrm{rcMMSE}_\mu(t)
  \le 4\mathrm{cMMSE}_\mu(t).$
\end{lemma}

\begin{proof}
Notice that for any real numbers $x,x',a$ we have
\begin{equation}
  1\wedge\frac{|x-x'|^2}{t^2}
  \le
  2\left(1\wedge\frac{|x-a|^2}{t^2}\right)
  +2\left(1\wedge\frac{|x'-a|^2}{t^2}\right).
  \label{eq:cap-quasitriangle-scalar}
\end{equation}
Indeed, $|x-x'|^2\le2|x-a|^2+2|x'-a|^2$, and
$1\wedge(2u+2v)\le2(1\wedge u)+2(1\wedge v)$.
Summing \cref{eq:cap-quasitriangle-scalar} over coordinates gives
\[
  \varphi_t(x,x')
  \le 2\varphi_t(x,a)+2\varphi_t(x',a).
\]
Hence, for every $Y_t$-measurable function $a(Y_t)$, the Nishimori
identity gives
\[
  \E\bigl[\varphi_t(X_t^{(1)},X_t^{(2)})\mid Y_t\bigr]=\E\bigl[\varphi_t(X_t^{(1)},X)\mid Y_t\bigr]
  \le
  4\E\bigl[\varphi_t(X,a(Y_t))\mid Y_t\bigr].
\]
Averaging over $Y_t$ and optimizing over $a$ concludes the proof.
\end{proof}

Equipped with Proposition \ref{prop:multii_I-MMSE_0} and Lemma \ref{lem:replica-center_1} we are now in a position to prove \cref{eq:rd-main}.

\begin{proposition}[Cauchy Bayes-risk area controls rate distortion]
\label{prop:rate-distortion}
For every $\mu\in\cP(T)$,
\begin{equation}
  \int_0^\infty\mathrm{RD}_\mu(t)\,dt
  \le
  20\int_0^\infty\mathrm{cMMSE}_\mu(t)\,dt.
  \label{eq:rate-distortion-proposition}
\end{equation}
\end{proposition}

\begin{proof}
Fix $t>0$ and consider the Cauchy observation
$Y_t=X+tZ$ with prior $X\sim\mu$.  Let $X_t^{(1)}$ be a draw from
$P_{X\mid Y_t}$ sampled conditionally independently of $X$ given
$Y_t$.
By Bayes' rule, $X_t^{(1)}$ has marginal
law $\mu$, and therefore $(X,X_t^{(1)})$ is a valid coupling.  Hence
\begin{equation}
  I(X;X_t^{(1)})\le I(X;Y_t)=I_\mu(t),
  \label{eq:replica-data-processing}
\end{equation}
where the inequality follows from the Markov chain
$X\to Y_t\to X_t^{(1)}$.
Also, by the Nishimori identity,
\begin{equation}
  \E\varphi_t(X,X_t^{(1)})=\mathrm{rcMMSE}_\mu(t),
  \label{eq:replica-distortion}
\end{equation}
Therefore, combining this with
\cref{prop:multii_I-MMSE_0},
\begin{align}
  \mathrm{RD}_\mu(t)
  &\le I_\mu(t)+\mathrm{rcMMSE}_\mu(t) \le
  \mathrm{rcMMSE}_\mu(t)+16\int_t^\infty\frac{\mathrm{cMMSE}_\mu(s)}s\,ds.
  \label{eq:rd-pointwise-tail}
\end{align}
Integrating \cref{eq:rd-pointwise-tail}, applying Tonelli's theorem,
and using \cref{lem:replica-center_1},
\begin{align*}
  \int_0^\infty\mathrm{RD}_\mu(t)\,dt
  &\le
  \int_0^\infty \mathrm{rcMMSE}_\mu(t)\,dt
  +16\int_0^\infty
    \int_t^\infty\frac{\mathrm{cMMSE}_\mu(s)}s\,ds\,dt
  \\
  &\le
  20\int_0^\infty \mathrm{cMMSE}_\mu(s)\,ds.
\end{align*}
This proves \cref{eq:rate-distortion-proposition}.
\end{proof}

\begin{proof}[Proof of \cref{thm:cauchy-areas}]
The first inequality is \cref{prop:first-area}, and the second is
\cref{prop:rate-distortion}.  Combining them gives
\[
  \int_0^\infty\mathrm{RD}_\mu(t)\,dt
  \le20\int_0^\infty\mathrm{cMMSE}_\mu(t)\,dt
  \le80\pi\,b(T).
\]
\end{proof}

\section{From the rate--distortion integral to distributional decomposition}
\label{sec:direct-rd-bernoulli}

The goal of this section is to prove the prescribed-law estimate
$\delta_T(\mu)\lesssim\int_0^\infty\mathrm{RD}_\mu(t)\,dt$ and then
use a minimax argument to derive $\Lambda(T)\lesssim R(T)$, which is
\cref{thm:rd-bernoulli}.  We state the two required lemmas, deduce the
theorem, and then prove the lemmas.  The prescribed-law estimate will also
supply the final comparison in \cref{sec:distributional-completion}.

\subsection{Two key lemmas}
We use the Gaussian width $W$ from \eqref{eq:gaussian-width}.  The
distributional functionals $G$ and $\delta_T$ were defined previously.

The first lemma bounds $\delta_T(\mu)$ by the rate--distortion
integral for every $\mu\in\cP(T)$.
\begin{lemma}
\label{lem:distributional-decomposition_0}
For every $\mu\in\cP(T)$,
\begin{equation}
  \delta_T(\mu)
  \le30\int_0^\infty\mathrm{RD}_\mu(t)\,dt.
  \label{eq:distributional-decomposition-bound}
\end{equation}
\end{lemma}

The second, a minimax lemma, places
$\sup_{\mu\in\cP(T)}\delta_T(\mu)$ in a form directly comparable with
the Bernoulli functional $\Lambda(T)$.
\begin{lemma}
\label{lem:minimax-decomposition_0}
For every finite $T\subset\R^n$,
\[
  \inf_{a:T\to\R^n}
  \left\{
    \max_{x\in T}\norm{a(x)}_1
    +W\bigl(\{x-a(x):x\in T\}\bigr)
  \right\}
  \le
  2\sup_{\mu\in\cP(T)}\delta_T(\mu).
\]
\end{lemma}

\subsection{Proof of the rate--distortion lower bound}

\begin{proof}[Proof of \cref{thm:rd-bernoulli}]
By \cref{lem:distributional-decomposition_0,lem:minimax-decomposition_0},
\[
  \inf_{a:T\to\R^n}
  \left\{
    \max_{x\in T}\norm{a(x)}_1
    +W\bigl(\{x-a(x):x\in T\}\bigr)
  \right\}
  \le
  60\sup_{\mu\in\cP(T)}
  \int_0^\infty\mathrm{RD}_\mu(t)\,dt
  =
  60R(T).
\]
Fix $\varepsilon>0$, and choose $a:T\to\R^n$ such that, with
\[
  A:=a(T),
  \qquad
  U:=\{x-a(x):x\in T\},
  \qquad
  M:=\max_{x\in T}\norm{a(x)}_1,
\]
we have $M+W(U)\le60R(T)+\varepsilon$.  Since $T\subseteq A+U$ and
$\sup_{z\in A}\|z\|_1=M$, the Gaussian majorizing-measure theorem
\cite{Talagrand1987} gives a universal constant $C_{\rm G}>0$ such that
\[
  \gamma_{2}(U,\norm{\cdot}_2)\le C_{\rm G}W(U).
\]
Consequently, \eqref{eq:familiar-bernoulli-functional} yields, with
$C_0:=\max\{1,C_{\rm G}\}$,
\[
  \Lambda(T)
  \le
  M+\gamma_2(U,\|\cdot\|_2)
  \le
  M+C_{\rm G}W(U)
  \le
  C_0\bigl(60R(T)+\varepsilon\bigr).
\]
Letting $\varepsilon\downarrow0$ proves
$R(T)\ge(60C_0)^{-1}\Lambda(T)$, and hence
\eqref{eq:rd-bernoulli}.
\end{proof}

\subsection{The distributional decomposition lemma}

We now prove \cref{lem:distributional-decomposition_0}, which supplies
the last comparison in the prescribed-law cycle.  We first prove a
discrete multiscale bound and then compare the resulting dyadic sum
with the rate--distortion integral.  For $k\in\mathbb Z$, set
\[
  r_k:=2^{-k}.
\]

The following construction is based on the same general
decomposition-from-chaining principle as
\cite[Proposition~4.3]{Talagrand1994} and
\cite[Theorem~3.1]{BednorzLatala2014}: a multiscale, chaining-type
description is converted into an $\ell_1$-plus-Gaussian decomposition.
Here the multiscale data are supplied by rate--distortion
self-couplings.

\begin{lemma}
\label{lem:clipped-rounding}
Let $\mu\in\cP(T)$, let $X\sim\mu$, and let $k_0<k_1$ be integers
such that
\[
  r_{k_0}\ge2\diam_\infty(T).
\]
Suppose that $U_{k_0},\ldots,U_{k_1}$ are conditionally independent
given $X$ and satisfy $U_k\sim\mu$ for every $k=k_0,\ldots,k_1$.
Assume in addition that $U_{k_0}$ is independent of $X$ and that
$U_{k_1}=X$.  For $k=k_0,\ldots,k_1$, set
\[
  d_k:=\E\varphi_{r_k}(X,U_k),
  \qquad
  q_k:=I(X;U_k).
\]
Then
\begin{equation}
  \delta_T(\mu)
  \le
  \frac{17}{2}\sum_{k=k_0+1}^{k_1-1}r_kd_k
  +\frac{9}{2\sqrt2}\sum_{k=k_0}^{k_1}r_k(d_k+q_k).
  \label{eq:clipped-rounding-conclusion}
\end{equation}
\end{lemma}

\begin{proof}
Let $\operatorname{clip}_R$ act coordinatewise by truncation to
$[-R,R]$, and define for $k=k_0,\ldots,k_1-1,$
\[
  \Delta_k:=
  \operatorname{clip}_{3r_k/2}(U_{k+1}-U_k),
  \qquad
  \widetilde V:=U_{k_0}+\sum_{k=k_0}^{k_1-1}\Delta_k.
\]
For $x\in\supp\mu$, define
\[
  V(x):=\E[\widetilde V\mid X=x],
  \qquad
  A(x):=x-V(x).
\]
On any $\mu$-null point of $T$, set $A(x)=0$ and $V(x)=x$.
Thus $x=A(x)+V(x)$ on $T$.

We first bound $A$.  Fix a coordinate $j=1,\ldots,n$, and define for $k=k_0,\ldots,k_1,$
$v_k:=U_{k,j}-X_j$. We call $k$ bad when $\abs{v_k}>r_k$, and good otherwise.

Notice that the endpoint indices $k_0,k_1$ are good by definition.  Now, suppose that $p,\ldots,q$ is any
maximal run of bad indices; that is, $p,\dots,q$ are bad and $p-1$ and $q+1$ are good.
Define for $k=k_0,\ldots,k_1-1$ and $j=1,\ldots,n,$
\begin{align}
 E_{k,j}:=
  (U_{k+1,j}-U_{k,j})-\Delta_{k,j}.
  \label{e71}
\end{align}
Now, for any feasible value of $k$, if $k$ and $k+1$ are good, then
\[
  \abs{U_{k+1,j}-U_{k,j}}
  \le r_k+r_{k+1}=\frac32r_k,
\]
so $E_{k,j}=0$.  
Then from \eqref{e71} we obtain
\begin{align}
  \left|\sum_{k=p-1}^{q}E_{k,j}\right|
  \le
  r_{p-1}+r_{q+1}+\frac32\sum_{k=p-1}^{q}r_k
  \le\frac{17}{2}r_p.
  \label{e72}
\end{align}
Since
$
  X_j-\widetilde V_j=\sum_{k=k_0}^{k_1-1}E_{k,j}
$,
we may decompose $\{k_0+1,\dots,k_1-1\}$ into disjoint blocks of consecutive bad indices (maximal bad runs), and use \eqref{e72} to obtain that for any $j=1,\ldots,n,$

\begin{align*}
  \abs{X_j-\widetilde V_j}
  &\le
  \frac{17}{2}\sum_{k=k_0+1}^{k_1-1}
  r_k\mathbf{1}_{\{\abs{X_j-U_{k,j}}>r_k\}}
\end{align*}
and hence
\begin{align}
  \norm{X-\widetilde V}_1
  \le
  \frac{17}{2}\sum_{k=k_0+1}^{k_1-1}
  r_k\sum_{j=1}^n
 \mathbf{1}_{\{\abs{X_j-U_{k,j}}>r_k\}}\le
  \frac{17}{2}\sum_{k=k_0+1}^{k_1-1}
  r_k\varphi_{r_k}(X,U_k).
  \label{e_73}
\end{align}
The last inequality holds because every indicated coordinate contributes
$1$ to $\varphi_{r_k}(X,U_k)$.  
Applying Jensen's inequality to \eqref{e_73} gives
\begin{equation}
  \E\norm{A(X)}_1
  = 
    \E\norm{X-\E[\widetilde V|X]}_1
  \le
  \frac{17}{2}\sum_{k=k_0+1}^{k_1-1}r_kd_k.
  \label{eq:clipped-l1}
\end{equation}

For the Gaussian part, the coordinatewise definition of clipping gives
\begin{equation}
  \norm{\Delta_k}_2^2
  \le
  \frac94r_k^2
  \left(
    \varphi_{r_k}(X,U_k)
    +\varphi_{r_{k+1}}(X,U_{k+1})
  \right).
  \label{eq:clipped-l2}
\end{equation}
To verify this coordinatewise, put
\[
  \alpha:=\frac{\abs{X_j-U_{k,j}}}{r_k},
  \qquad
  \beta:=\frac{\abs{X_j-U_{k+1,j}}}{r_{k+1}}.
\]
If $\alpha>1$ or $\beta>1$, the corresponding coordinate on the right-hand
side of \eqref{eq:clipped-l2} is at least $\frac94r_k^2$, which bounds
$\abs{\Delta_{k,j}}^2$ by clipping.  If $\alpha,\beta\le1$, then
\[
  \abs{\Delta_{k,j}}^2
  \le
  \abs{U_{k+1,j}-U_{k,j}}^2
  \le
  r_k^2\left(\alpha+\frac\beta2\right)^2
  \le
  \frac94r_k^2(\alpha^2+\beta^2).
\]
This proves \eqref{eq:clipped-l2}.

Take now an arbitrary coupling of $X\sim\mu$ with a standard Gaussian
$G_0$, as in the pushforward representation of
$G(V_{\#}\mu)$.  Conditionally on
$X$, sample the $U_k$'s according to their original product kernel,
independently of $G_0$.  Then
$\E[\widetilde V\mid X,G_0]=V(X)$, so
\[
  \E\ip{V(X)}{G_0}
  =
  \E\ip{\widetilde V}{G_0}
  =
  \sum_{k=k_0}^{k_1-1}\E\ip{\Delta_k}{G_0};
\]
the $U_{k_0}$-term vanishes because $U_{k_0}$ is independent of $G_0$.
Moreover, data processing and conditional independence give
\[
  I(G_0;U_k,U_{k+1})
  \le
  I(X;U_k,U_{k+1})
  \le q_k+q_{k+1}.
\]
The Gaussian mean--information inequality gives
\[
  \E\norm{\E[G_0\mid U_k,U_{k+1}]}_2^2
  \le2I(G_0;U_k,U_{k+1});
\]
indeed, this follows by applying the following inequality
\[
  \KL(Q\|N(0,I_n))
  \ge\frac12\norm{\E_QG_0}_2^2,
  \quad 
  \forall Q
\]
with $Q=P_{G_0\mid U_k,U_{k+1}}$, and then taking the expectation with
respect to $P_{U_k,U_{k+1}}$. Since $\Delta_k$ is a function of
$(U_k,U_{k+1})$, Cauchy--Schwarz and \eqref{eq:clipped-l2} yield
\begin{align*}
  \E\ip{\Delta_k}{G_0}
  &=
  \E\ip{\Delta_k}{\E[G_0\mid U_k,U_{k+1}]}\\
  &\le
  \bigl(\E\norm{\Delta_k}_2^2\bigr)^{1/2}
  \bigl(\E\norm{\E[G_0\mid U_k,U_{k+1}]}_2^2\bigr)^{1/2}\\
  &\le
  \frac{3}{\sqrt2}r_k
  \sqrt{(d_k+d_{k+1})(q_k+q_{k+1})}\\
  &\le
  \frac{3}{2\sqrt2}r_k
  (d_k+d_{k+1}+q_k+q_{k+1}).
\end{align*}
Writing $s_k:=d_k+q_k$, using $r_{k-1}=2r_k$, summing in $k$, and
then taking the supremum over couplings of $X$ and $G_0$, we obtain
\begin{align}
  G(V_{\#}\mu)
  &\le
  \frac{3}{2\sqrt2}
  \sum_{k=k_0}^{k_1-1}r_k(s_k+s_{k+1})\notag\\
  &\le
  \frac{9}{2\sqrt2}\sum_{k=k_0}^{k_1}r_ks_k.
  \label{eq:clipped-gaussian}
\end{align}
Combining \eqref{eq:clipped-l1} and \eqref{eq:clipped-gaussian} proves
\eqref{eq:clipped-rounding-conclusion}.
\end{proof}

We are now in a position to prove \cref{lem:distributional-decomposition_0}.

\begin{proof}[Proof of \cref{lem:distributional-decomposition_0}.]
The function $t\mapsto\mathrm{RD}_\mu(t)$ is clearly nonincreasing.  Integrating over
each interval $[r_{k+1},r_k], k \in \mathbb{Z}$ therefore gives
\begin{equation}
  \sum_{k\in\ZZ}r_k\mathrm{RD}_\mu(r_k)
  \le
  2\int_0^\infty\mathrm{RD}_\mu(t)\,dt.
  \label{eq:dyadic-rate-distortion_0}
\end{equation}
Removing $\mu$-null points from $T$ changes neither side of
\eqref{eq:distributional-decomposition-bound}, so assume
$\supp\mu=T$.  Fix integers $k_0<k_1$, with $k_0$
sufficiently negative that
$r_{k_0}\ge2\diam_\infty(T)$.  The finite self-coupling polytope is
compact and the objective in \eqref{eq:rate-distortion} is continuous,
so for every $k_0<k<k_1$ we may choose a self-coupling $(X,U_k)$
attaining $\mathrm{RD}_\mu(r_k)$.  Combine the corresponding kernels
conditionally independently given $X$, and complete the family with
an independent copy $U_{k_0}$ and with $U_{k_1}=X$.
For $k_0 \leq k \leq k_1$, let $d_k$ and $q_k$ be the quantities associated with this family in
\cref{lem:clipped-rounding}.  At every interior index,
\[
  d_k+q_k=\mathrm{RD}_\mu(r_k),
\]
whereas the endpoint quantities satisfy
\[
  q_{k_0}=0,
  \qquad
  d_{k_1}=0,
  \qquad
  q_{k_1}=H(\mu).
\]
Since $d_k\le d_k+q_k$ and
$\frac{17}{2}+\frac{9}{2\sqrt2}<15$,
Lemma~\ref{lem:clipped-rounding} gives
\[
  \delta_T(\mu)
  \le
  15\left[
    \sum_{k=k_0+1}^{k_1-1}r_k\mathrm{RD}_\mu(r_k)
    +r_{k_0}d_{k_0}+r_{k_1}H(\mu)
  \right].
\]
If $X'$ is an independent copy of $X$, then since $T$ is bounded,
\[
  r_{k_0}d_{k_0}
  \le
  \frac{\E\norm{X-X'}_2^2}{r_{k_0}}
  \longrightarrow0
  \qquad(k_0\to-\infty),
\]
while $r_{k_1}H(\mu)\to0$ as $k_1\to+\infty$.  Letting the endpoints
tend to $-\infty$ and $+\infty$, respectively, and using
\eqref{eq:dyadic-rate-distortion_0} proves the claim.
\end{proof}

\subsection{The minimax decomposition lemma}

We next prove \cref{lem:minimax-decomposition_0}.  This step converts
the prior-wise decompositions furnished by
\cref{lem:distributional-decomposition_0} into the single set-level
decomposition required for \cref{thm:rd-bernoulli}.

\begin{proof}[Proof of \cref{lem:minimax-decomposition_0}]
For $a:T\to\R^n$, put $v_a:=\operatorname{id}_T-a$ and define
\[
  J(a,\mu)
  :=
  \sum_{x\in T}\mu(x)\norm{a(x)}_1
  +G\bigl((v_a)_{\#}\mu\bigr).
\]
Then $\delta_T(\mu)=\inf_a J(a,\mu)$.  By the pushforward
representation of $G$, for fixed $\mu$,
\[
  G\bigl((v_a)_{\#}\mu\bigr)
  =
  \sup_{\substack{X\sim\mu,\ G_0\sim N(0,I_n)\\
                  \text{jointly coupled}}}
  \mathbb E\ip{v_a(X)}{G_0}.
\]
The right-hand side is a supremum of continuous affine functions of
$a$ and is finite everywhere, so it is convex and continuous on the
finite-dimensional space $(\mathbb R^n)^T$.  Hence
$a\mapsto J(a,\mu)$ is convex and lower semicontinuous.

For fixed $a$, the map $\mu\mapsto J(a,\mu)$ is concave because
mixtures of feasible couplings are feasible.  It is upper
semicontinuous as well.  Indeed, if $\mu_k\to\mu$, near-optimal
couplings along a subsequence realizing the limsup are tight because
their Gaussian marginal is fixed.  Any weak limit is feasible for
$\mu$, while
$|\ip{v_a(X)}{G_0}|\le M_a\|G_0\|_2$, where
$M_a:=\max_{x\in T}\|v_a(x)\|_2$, gives uniform integrability.
Passing to the limit in the coupling objective and using the
continuity of the first term in $J(a,\mu)$ proves upper
semicontinuity.

The set $\cP(T)$ is compact and convex, while $(\R^n)^T$ is convex.
Sion's minimax theorem therefore yields
\begin{equation}
  \inf_a\sup_{\mu\in\cP(T)}J(a,\mu)
  =
  \sup_{\mu\in\cP(T)}\inf_aJ(a,\mu)
  =
  \sup_{\mu\in\cP(T)}\delta_T(\mu).
  \label{eq:direct-sion_0}
\end{equation}

Fix $a$.  Taking $\mu$ to be a point mass at a maximizer of
$\|a(x)\|_1$ gives
\[
  \sup_{\mu\in\cP(T)}J(a,\mu)
  \ge
  \max_{x\in T}\norm{a(x)}_1.
\]
Let $x(G_0)$ be a measurable maximizer of $\ip{v_a(x)}{G_0}$, and let
$\mu_a$ be its law.  The coupling $(v_a(x(G_0)),G_0)$ then gives
\[
  G\bigl((v_a)_{\#}\mu_a\bigr)
  \ge
  \E\max_{x\in T}\ip{v_a(x)}{G_0}
  =
  W\bigl(v_a(T)\bigr).
\]
It follows that
\[
  \max_{x\in T}\norm{a(x)}_1+W\bigl(v_a(T)\bigr)
  \le
  2\sup_{\mu\in\cP(T)}J(a,\mu).
\]
Taking the infimum over $a$ and using \eqref{eq:direct-sion_0} proves
the lemma.
\end{proof}

\section{Proof of the distributional theorem}
\label{sec:distributional-completion}

The goal of this section is to complete
\cref{thm:distributional-bernoulli-rd}.  Type lifting, meaning passage
from a law to the uniform law on sequences with that empirical
distribution, combines with \cref{thm:cauchy-areas} to give
$\int_0^\infty\mathrm{RD}_\mu(t)\,dt\lesssim B(\mu)$, elementary
Gaussianization gives $B(\mu)\lesssim\delta_T(\mu)$, and
\cref{lem:distributional-decomposition_0} supplies the remaining
comparison
$\delta_T(\mu)\lesssim\int_0^\infty\mathrm{RD}_\mu(t)\,dt$.
Taking the supremum over $\mu$ and invoking
\cref{thm:rd-bernoulli} then completes the set-level assertion.

\begin{lemma}[Lifting rate--distortion inequality]
\label{lem:lifting-rd-inequality}
For every finite $T\subset\R^n$ and every $\mu\in\cP(T)$,
\[
  B(\mu)
  \gtrsim
  \int_0^\infty \mathrm{RD}_{\mu}(t)\,\dif t.
\]
\end{lemma}

\begin{proof}
Call $\mu\in\cP(T)$ rational if $\mu(x)\in\mathbb Q$ for every
$x\in T$.  We first treat rational $\mu$.  Call a positive integer
$N$ compatible with $\mu$ if $N\mu(x)$ is an integer for every $x$.
For a sequence $\mathbf x=(x_1,\ldots,x_N)\in T^N$, its empirical
measure
\[
  L_N(\mathbf x):=\frac1N\sum_{j=1}^N\delta_{x_j}
\]
is called its \emph{type} in information theory.  The corresponding
\emph{exact type class} is the fiber
\[
  \mathcal T_N(\mu)
  :=
  \{\mathbf x\in T^N:L_N(\mathbf x)=\mu\}.
\]
Thus $\mathcal T_N(\mu)$ consists of the sequences containing exactly
$N\mu(x)$ copies of each $x$.  Let $\mu_N$ be its uniform law,
equivalently, the conditional law of an i.i.d.\ $\mu$-sample given
$L_N=\mu$.

Applied to the Bernoulli process, the type lifting identity
\cite[Lemma~5]{Liu2025}; see also
\cite[Proposition~3.1]{van2025subgaussian}, gives
\begin{equation}
  B(\mu)
  =
  \lim_{\substack{N\to\infty\\N\text{ compatible with }\mu}}
  \frac1N b\bigl(\mathcal T_N(\mu)\bigr).
  \label{eq:type-lifting}
\end{equation}

We next quantify the small entropy loss caused by conditioning on the
type.  Set
\[
  m:=|\supp(\mu)|,
  \qquad
  r_N:=NH(\mu)-\log|\mathcal T_N(\mu)|.
\]
If $K$ is the count vector of an i.i.d.\ $\mu$-sample and
$n_x:=N\mu(x)$ for $x\in\supp(\mu)$, then
\[
  \mathbb P(K=k)
  =
  \frac{N!}{N^N}
  \prod_{x\in\supp(\mu)}\frac{n_x^{k_x}}{k_x!}.
\]
Since $k\mapsto n^k/k!$ is maximized at $k=n$ (and at $n-1$),
$K=(n_x)_x$ is a mode.  There are at most $(N+1)^{m-1}$ possible
count vectors, while
\[
  \mathbb P\{L_N=\mu\}
  =
  |\mathcal T_N(\mu)|e^{-NH(\mu)}
  =
  e^{-r_N}.
\]
Consequently,
\begin{equation}
  0\le r_N\le(m-1)\log(N+1).
  \label{eq:type-defect}
\end{equation}

Regard $\mathcal T_N(\mu)$ as a subset of
$(\mathbb R^n)^N\cong\mathbb R^{nN}$ and interpret
$\mathrm{RD}_{\mu_N}(t)$ according to \eqref{eq:rate-distortion} in
this ambient space.  In particular,
\[
  \varphi_t(\mathbf x,\mathbf u)
  =
  \sum_{j=1}^N\varphi_t(x_j,u_j).
\]
Let $(\mathbf X,\mathbf U)$ be any self-coupling of $\mu_N$, and let
$J$ be independent and uniform on $\{1,\ldots,N\}$.  Then
$X_J,U_J\sim\mu$, and conditional-entropy subadditivity gives
\[
\begin{aligned}
  I(\mathbf X;\mathbf U)
  &=
  NH(\mu)-r_N-H(\mathbf X\mid\mathbf U)\\
  &\ge
  NH(\mu)-r_N-\sum_{j=1}^N H(X_j\mid U_j)\\
  &=
  NH(\mu)-r_N-NH(X_J\mid U_J,J)\\
  &\ge
  N I(X_J;U_J)-r_N.
\end{aligned}
\]
Moreover,
$\mathbb E\varphi_t(\mathbf X,\mathbf U)
=N\mathbb E\varphi_t(X_J,U_J)$.  Taking the infimum over the
self-coupling gives
\begin{equation}
  \mathrm{RD}_{\mu_N}(t)
  \ge
  N\mathrm{RD}_{\mu}(t)-r_N.
  \label{eq:type-rd-transfer}
\end{equation}

Put $D:=\diam_2(T)$ and
$s_N:=\sqrt{N/(r_N+1)}$.  The independent self-coupling gives
$\mathrm{RD}_\mu(t)\le D^2/t^2$.  We therefore integrate
\eqref{eq:type-rd-transfer} only up to the finite cutoff $s_N$, use
the preceding bound for the tail, and apply \cref{thm:cauchy-areas}:
\[
\begin{aligned}
  \int_0^\infty\mathrm{RD}_\mu(t)\,\dif t
  &\le
  \frac1N\int_0^\infty\mathrm{RD}_{\mu_N}(t)\,\dif t
  +\frac{s_Nr_N}{N}
  +\frac{D^2}{s_N}\\
  &\le
  \frac{80\pi}{N}b\bigl(\mathcal T_N(\mu)\bigr)
  +(1+D^2)\sqrt{\frac{r_N+1}{N}}.
\end{aligned}
\]
Letting $N\to\infty$ through compatible integers and using
\eqref{eq:type-lifting} and \eqref{eq:type-defect} proves the lemma for
rational $\mu$.

For general $\mu$, approximate it by rational laws.  If
$d:=\frac12\sum_{x\in T}|\nu(x)-\nu'(x)|$, a maximal-coupling kernel
gives the first of the following bounds, and data processing gives the
second:
\[
  |B(\nu)-B(\nu')|
  \le \diam_1(T)d,
  \qquad
  |\mathrm{RD}_\nu(t)-\mathrm{RD}_{\nu'}(t)|
  \le
  2d\min\left\{n,\frac{\diam_2(T)^2}{t^2}\right\}.
\]
Indeed, apply the kernel to the $T$-valued variable in the first
coupling problem and independently to both variables in the second;
each variable changes with probability $d$.  The bound on the right is
integrable over $(0,\infty)$, so both $B(\nu)$ and
$\int_0^\infty\mathrm{RD}_\nu(t)\,\dif t$ are continuous in $\nu$.
The rational case now completes the proof.
\end{proof}

\begin{proof}[Proof of \cref{thm:distributional-bernoulli-rd}]
By \cref{lem:lifting-rd-inequality},
\[
  \int_0^\infty\mathrm{RD}_{\mu}(t)\,\dif t
  \lesssim
  B(\mu).
\]
For the reverse comparison, let $\nu$ be a finitely supported
probability measure on $\R^n$ and fix a coupling of $Y\sim\nu$ and
$\varepsilon\sim\mathrm{Rad}^{\otimes n}$.  If
$g_1,\ldots,g_n$ are independent standard Gaussian variables,
independent of $(Y,\varepsilon)$, then
$\widetilde G:=(\varepsilon_i|g_i|)_{i=1}^n\sim N(0,I_n)$ and
\[
  \mathbb E\langle Y,\widetilde G\rangle
  =
  \sqrt{\frac{2}{\pi}}\,
  \mathbb E\langle Y,\varepsilon\rangle.
\]
Taking the supremum over couplings gives
$B(\nu)\le\sqrt{\pi/2}\,G(\nu)$.

Now fix $a:T\to\R^n$ and put
$\nu=(\operatorname{id}_T-a)_{\#}\mu$.  Since
$\|\varepsilon\|_\infty=1$, the preceding comparison gives
\[
\begin{aligned}
  B(\mu)
  &\le
  \mathbb E\|a(X)\|_1+B(\nu) \\
  &\le
  \sqrt{\frac{\pi}{2}}
  \left\{
    \mathbb E\|a(X)\|_1+G(\nu)
  \right\}.
\end{aligned}
\]
Taking the infimum over $a$ yields
$B(\mu)\le\sqrt{\pi/2}\,\delta_T(\mu)$.  Combining this with
\cref{lem:distributional-decomposition_0} gives
\[
  \int_0^\infty\mathrm{RD}_{\mu}(t)\,\dif t
  \lesssim
  B(\mu)
  \lesssim
  \delta_T(\mu)
  \lesssim
  \int_0^\infty\mathrm{RD}_{\mu}(t)\,\dif t,
\]
so the prescribed-law assertion follows.  Taking the supremum over
$\mu\in\cP(T)$, using \eqref{eq:RD-area} and the identity
$\sup_{\mu\in\cP(T)}B(\mu)=b(T)$ established above, gives
$b(T)\asymp R(T)$.  The elementary bound
$b(T)\lesssim\Lambda(T)$, together with \cref{thm:rd-bernoulli},
therefore yields $b(T)\asymp R(T)\asymp\Lambda(T)$.
This proves the set-level assertion.
\end{proof}

\appendix

\section{A Cauchy information--estimation inequality}\label{sec:i-MMSE}

We prove the Cauchy information--estimation inequality, \cref{prop:multii_I-MMSE_0}, used in
\cref{sec:rate-distortion}.  It is the analytic input for
\eqref{eq:rd-main}, the second comparison in \cref{thm:cauchy-areas}.
\subsection{A scalar Cauchy information--estimation inequality}
\label{sec:scalar-harmonic}

First, notice that for any $x \in \R$ the PDF of $x+tZ,$ where $Z$ follows a Cauchy distribution is the Poisson kernel
\begin{equation}
  c_t(y-x)
  =
  \frac1\pi\frac{t}{t^2+(y-x)^2}, y \in \R.
  \label{eq:poisson-kernel}
\end{equation}
It is easy to check that for any $x \in \R,$  $c_t(y-x)$ as a function of $(y,t)\in\R\times(0,\infty)$ is harmonic, i.e., it satisfies for all $y,x \in \R$ and $t >0,$
\begin{equation}
  (\partial_y^2+\partial_t^2)c_t(y-x)=0.
  \label{eq:poisson-harmonic}
\end{equation}

The following one-dimensional information--estimation identity and
consequence hold.
\begin{lemma}[Scalar Cauchy information-estimation identity]
\label{lem:scalar-curvature}
Let $U$ be a finitely supported real random variable, let $S$ be an arbitrarily distributed random variable, and let $Z$ be standard Cauchy and independent of
$(U,S)$.  Define
\begin{equation}
  F(t):=I(U;U+tZ\mid S).
  \label{eq:conditional-information}
\end{equation}
Then $F$ is decreasing and convex on $(0,\infty)$, with
$\lim_{t\to\infty}F(t)=0$.
Moreover, let $U_t^{(1)},U_t^{(2)}$ be i.i.d.\ draws from the
posterior of $U$ given $(S,U+tZ)$.  Then
\begin{equation}
  F''(t)
  \le
  \frac{2}{t^2}
  \E\left[
    1\wedge
    \frac{|U_t^{(1)}-U_t^{(2)}|^2}{t^2}
  \right].
  \label{eq:scalar-curvature-bound}
\end{equation}
Consequently,
\begin{equation}
  -F'(t)
  \le
  2\int_t^\infty
  \frac{1}{s^2}
  \E\left[
    1\wedge
    \frac{|U_s^{(1)}-U_s^{(2)}|^2}{s^2}
  \right]ds.
  \label{eq:scalar-tail-derivative}
\end{equation}
\end{lemma}

\begin{proof}
We first prove the result when $S$ is deterministic.  Fix a finite
prior $\mu$ for $U$, and define
\[
  u_x(t,y):=c_t(y-x),
  \qquad
  p(t,y):=\int_{-\infty}^{+\infty} u_x(t,y)\,\mu(dx).
\]
Both $u_x$ and $p$ are positive harmonic functions of $(y,t)$.  Also
$I(U;U+tZ)$ equals
\begin{equation}
  F(t)
  =
  \int_{-\infty}^{+\infty}\!\int_{-\infty}^{+\infty}
  u_x(t,y)\log\frac{u_x(t,y)}{p(t,y)}
  \,dy\,\mu(dx)
  =
    \int_{-\infty}^{+\infty}\!\int_{-\infty}^{+\infty}
    \Psi(u_x(t,y),p(t,y))
  \,dy\,\mu(dx)
  \label{eq:mutual-info-density}
\end{equation}
where we defined, for positive variables $u,v$,  $\Psi(u,v):=u\log(u/v)$.  Direct calculations then give,
\[
  \Psi_{uu}=\frac1u,
  \qquad
  \Psi_{uv}=-\frac1v,
  \qquad
  \Psi_{vv}=\frac{u}{v^2}.
\]
Let $\nabla$ denote the two-dimensional gradient in $(y,t)$.  Setting
$u=u_x(t,y)$ and $v=p(t,y)$, which are harmonic in $(y,t)$, gives
\begin{align}
  (\partial_y^2+\partial_t^2)\Psi(u,v)
  &=
  \frac{|\nabla u|^2}{u}
  -2\frac{\nabla u\cdot\nabla v}{v}
  +u\frac{|\nabla v|^2}{v^2}
  \notag\\
  &=
  u\,|\nabla\log u-\nabla\log v|^2,
  \label{eq:harmonic-relative-entropy}
\end{align}

Let $J=[a,b]\Subset(0,\infty)$ and write
$m_1:=\int x\,\mu(dx)$.  Since $\mu$ has finite support, uniformly for
$t\in J$ and $x\in\operatorname{supp}\mu$, as $|y|\to\infty$,
\begin{align*}
  u_x(t,y)
  &=
  \frac{t}{\pi y^2}
  \left(1+\frac{2x}{y}+O_J(|y|^{-2})\right),
  \\
  p(t,y)
  &=
  \frac{t}{\pi y^2}
  \left(1+\frac{2m_1}{y}+O_J(|y|^{-2})\right).
\end{align*}
Hence
\[
  \log\frac{u_x(t,y)}{p(t,y)}
  =
  \frac{2(x-m_1)}{y}+O_J(|y|^{-2}).
\]
Differentiating the finite rational expressions in $t$ and $y$ gives
\[
  \bigl|\partial_t^j\Psi(u_x(t,y),p(t,y))\bigr|
  \le C_J(1+|y|)^{-3},
  \qquad j=0,1,2,
\]
and
\[
  \bigl|\partial_y\Psi(u_x(t,y),p(t,y))\bigr|
  \le C_J(1+|y|)^{-4}.
\]
On bounded $y$-intervals these bounds follow from smoothness and the
strict positivity of $p$, uniformly for $t\in J$.  Thus
differentiation under the integral sign is justified by dominated
convergence.  Also, for every $R>0$,
\[
  \int_{-R}^R\partial_y^2\Psi(u_x,p)\,dy
  =
  \partial_y\Psi(u_x,p)(R)-\partial_y\Psi(u_x,p)(-R),
\]
which tends to zero as $R\to\infty$ by the last bound.  Therefore,
after integrating \eqref{eq:harmonic-relative-entropy} in $y$, the
left-hand side is precisely $F''(t)$.  Writing $Y_t:=U+tZ$, we obtain

\begin{align*}
    F''(t)
    &=
    \int_{-\infty}^{+\infty}
    \int_{-\infty}^{+\infty}
    u_x(t,y)
    \left|
        \nabla\log u_x(t,y)
        -
        \nabla\log p(t,y)
    \right|^2
    dy\,\mu(dx),
\end{align*}
or, equivalently,
\begin{equation}
  F''(t)
  =
  \E\left|
    \nabla\log c_t(Y_t-U)
    -\nabla\log p(t,Y_t)
  \right|^2,
  \label{eq:curvature-score-variance-1}
\end{equation}

A direct differentiation implies
\begin{equation}
  \nabla\log p(t,Y_t)
  =
  \E\left[
    \nabla\log c_t(Y_t-U)
    \mid Y_t
  \right].
  \label{eq:mixture-score}
\end{equation}
In particular, \cref{eq:curvature-score-variance-1} is a conditional
variance, which proves that $F$ is convex.  If
$U_t^{(1)},U_t^{(2)}$ are two posterior replicas of $U$ given
$Y_t$, then
\begin{equation}
  F''(t)
  =
  \frac12\E\left|
    \nabla\log c_t(Y_t-U_t^{(1)})
    -\nabla\log c_t(Y_t-U_t^{(2)})
  \right|^2.
  \label{eq:curvature-score-variance-2}
\end{equation}

We now compute the integrand. Setting $r=y-x$ and $v=r/t$, direct differentiation of the Poisson kernel~\cref{eq:poisson-kernel} gives 
\begin{equation}
  t\nabla_{(y,t)}\log c_t(r)
  =
  V(v)
  :=
  \left(
    -\frac{2v}{1+v^2},
    \frac{v^2-1}{1+v^2}
  \right).
  \label{eq:cauchy-score-circle}
\end{equation}
Note that the vector $V(v)$ lies on the unit circle.  A direct calculation
then gives
\begin{equation}
  \|V(v)-V(w)\|^2
  =
  \frac{4(v-w)^2}{(1+v^2)(1+w^2)}
  \le
  4\bigl(1\wedge|v-w|^2\bigr).
  \label{eq:circle-distance-cap}
\end{equation}
Apply \cref{eq:cauchy-score-circle} with
\[
  v:=\frac{Y_t-U_t^{(1)}}{t},
  \qquad
  w:=\frac{Y_t-U_t^{(2)}}{t}.
\]
Since
\[
  v-w=\frac{U_t^{(2)}-U_t^{(1)}}{t},
\]
substituting \cref{eq:circle-distance-cap} into
\cref{eq:curvature-score-variance-2} proves
\cref{eq:scalar-curvature-bound}.

For $0<t_1<t_2$, let $Z_1,Z_2$ be independent standard Cauchy random
variables, independent of $U$.  Cauchy stability gives
\[
  t_1Z_1+(t_2-t_1)Z_2\ \stackrel d=\ t_2Z.
\]
Thus
\[
  U\longrightarrow U+t_1Z_1
  \longrightarrow U+t_1Z_1+(t_2-t_1)Z_2
\]
is a Markov chain, and data processing gives $F(t_2)\le F(t_1)$.
To see that $\lim_{t\to\infty}F(t)=0$, fix $u_0$ in the support of
$U$.  The information-radius identity gives
\begin{align*}
  F(t)
  &=
  \E_U D_{\mathrm{KL}}\bigl(
    P_{U+tZ\mid U}\,\|\,P_{u_0+tZ}
  \bigr)
  -
  D_{\mathrm{KL}}\bigl(
    P_{U+tZ}\,\|\,P_{u_0+tZ}
  \bigr)
  \\
  &\le
  \E_U D_{\mathrm{KL}}\bigl(
    P_{U+tZ\mid U}\,\|\,P_{u_0+tZ}
  \bigr)
  \\
  &=
  \E\log\left(
    \frac{1+(Z+(U-u_0)/t)^2}{1+Z^2}
  \right)
  \le
  \frac{\E|U-u_0|}{t}
  \longrightarrow0.
\end{align*}
The last inequality uses that $z\mapsto\log(1+z^2)$ is
$1$-Lipschitz.

Since $F$ is nonnegative, non-increasing, and convex, it must also hold that $\lim_{t \to \infty}F'(t)=0$.  Therefore
\[
  -F'(t)=\int_t^{+\infty} F''(s)\,ds,
\]
and \cref{eq:scalar-tail-derivative} follows from
\cref{eq:scalar-curvature-bound}.

We now pass to general side information.  Let
$A:=\operatorname{supp}(U)$ and $m:=|A|$.  For a probability vector
$q$ on $A$, let $F_q(t)$ and $R_q(t)$ denote, respectively, the mutual
information and the posterior-replica risk in the deterministic result
with prior $q$.  If
$q_S(u):=\mathbb P(U=u\mid S)$, then
\[
  F(t)=\mathbb E F_{q_S}(t).
\]
The deterministic result and the bound $R_q(t)\le1$ give, uniformly
in $q$,
\[
  0\le F_q(t)\le\log m,
  \qquad
  0\le F_q''(t)\le\frac{2}{t^2},
  \qquad
  0\le-F_q'(t)\le\frac{2}{t}.
\]
Consequently, dominated convergence on every compact subinterval of
$(0,\infty)$ yields
\[
  F'(t)=\mathbb E F_{q_S}'(t),
  \qquad
  F''(t)=\mathbb E F_{q_S}''(t).
\]
Moreover, $F_q(t)\to0$ for every $q$, so dominated convergence also
gives $F(t)\to0$.  Conditional on $S$, the replicas in the
deterministic result have precisely the conditional law of two
independent draws from $P_{U\mid S,U+tZ}$.  Averaging the deterministic
curvature bound therefore proves \cref{eq:scalar-curvature-bound}; the
same convexity and endpoint argument proves
\cref{eq:scalar-tail-derivative}.
\end{proof}

\subsection{The product Cauchy information--estimation inequality}
\label{sec:product-extension}

The main result of this subsection immediately implies
\cref{prop:multii_I-MMSE_0}.  Recall the replica cMMSE
$\mathrm{rcMMSE}_\mu(s)$ from
\cref{eq:diagonal-coordinate-risk}.

\begin{proposition}\label{prop:multii_I-MMSE}
For every $\mu\in\cP(T)$ and $t>0$,
\begin{align}
  -I_\mu'(t) \le 4
  \int_t^\infty\frac{\mathrm{rcMMSE}_\mu(s)}{s^2}\,ds.
\end{align}
In particular,
\begin{align}
  I_\mu(t) \le
4\int_t^\infty\frac{\mathrm{rcMMSE}_\mu(s)}{s}\,ds\le
16\int_t^\infty\frac{\mathrm{cMMSE}_\mu(s)}{s}\,ds.
\end{align}
\end{proposition}
\subsubsection{A multi-scale Cauchy additive model and first steps}

To prove \cref{prop:multii_I-MMSE}, we first consider a multiscale
version of the Cauchy additive channel.

Let $X\sim\mu$, and let $Z_1,\ldots,Z_n$ be independent standard
Cauchy random variables, independent of $X$.  For
$\bm t=(t_1,\ldots,t_n)\in[0,\infty)^n$, define
\begin{equation}
  Y_{\bm t}:=(X_i+t_iZ_i)_{i=1}^n,
  \qquad
  I_\mu(\bm t):=I(X;Y_{\bm t}).
  \label{eq:mixed-scale-channel}
\end{equation}
Thus $Y_{\bm t,i}=X_i+t_iZ_i$, with
$Y_{\bm t,i}=X_i$ when $t_i=0$.  Derivatives of
$I_\mu(\bm t)$ below are taken only at points of $(0,\infty)^n$.
The one-scale notation is consistent with this definition:
\[
  I_\mu(t)=I_\mu(t\mathbf1).
\]
Fix $i\in[n]$.  By the chain rule, if $Y_{\bm t,-i}$ denotes
$(Y_{\bm t,j})_{j\ne i}$, then
\begin{align}
  I_\mu(\bm t)
  &=I(X;Y_{\bm t,-i})
    +I(X;Y_{\bm t,i}\mid Y_{\bm t,-i})
  \notag\\
  &=I(X;Y_{\bm t,-i})
    +I(X_i;Y_{\bm t,i}\mid Y_{\bm t,-i}),
  \label{eq:coordinate-chain-rule}
\end{align}
where the last equality holds because $Y_{\bm t,i}$ depends on $X$ only through $X_i$.  Notice that the first term
does not depend on $t_i$, and therefore
\[
  \partial_{t_i}I_\mu(\bm t)
  =
  \partial_{t_i}I(X_i;Y_{\bm t,i}\mid Y_{\bm t,-i}).
\]

For a scale vector $\bm r$ and a coordinate $i$ with $r_i>0$, let
$X_{\bm r}^{(1)},X_{\bm r}^{(2)}$ be conditionally independent draws
from $P_{X\mid Y_{\bm r}}$, and define
\begin{equation}
  \mathrm{rcMMSE}_{i,\mu}(\bm r)
  :=
  \mathbb E\left[
    1\wedge
    \frac{
      |X_{\bm r,i}^{(1)}-X_{\bm r,i}^{(2)}|^2
    }{r_i^2}
  \right]
  \label{eq:mixed-coordinate-risk}
\end{equation}
For $i\in[n]$ and $s>0$, write
\[
  \bm t^{\,i,s}
  :=(t_1,\ldots,t_{i-1},s,t_{i+1},\ldots,t_n).
\]
In particular,
\[
  \mathrm{rcMMSE}_{i,\mu}(s\mathbf1)
  =
  \mathrm{rcMMSE}_{i,\mu}(s).
\]

Applying \cref{lem:scalar-curvature} with $U=X_i$ and side
information $S=Y_{\bm t,-i}$, and using the independence of $Z_i$
from $(X_i,S)$, gives
\begin{equation}
  -\partial_{t_i}I_\mu(\bm t)
  =
  -\partial_{t_i}I(X_i;Y_{\bm t,i}\mid Y_{\bm t,-i})
  \le
  2\int_{t_i}^\infty
  \frac{\mathrm{rcMMSE}_{i,\mu}(\bm t^{\,i,s})}{s^2}\,ds.
  \label{eq:mixed-partial-bound}
\end{equation}

The following data-processing lemma compares
$\mathrm{rcMMSE}_{i,\mu}(\bm t^{\,i,s})$, for
$\bm t=(t,\ldots,t)$, with the one-scale quantity
$\mathrm{rcMMSE}_{i,\mu}(s)$ from
\eqref{eq:diagonal-coordinate-risk}.

\begin{lemma}
For every $0\le t\le s$ and $i\in[n]$,
    \begin{equation}
  \mathrm{rcMMSE}_{i,\mu}((t,t,\ldots,t)^{i,s})
  \le
  2\mathrm{rcMMSE}_{i,\mu}(s).
  \label{eq:mixed-to-diagonal-risk}
\end{equation}
\end{lemma}

\begin{proof}
    
For squared loss, the claim follows from the usual data-processing
inequality for MMSE and the stability of Cauchy noise; see, e.g.,
\cite[Section~V]{wu2011functional}.  For the capped quadratic loss, we
use a quadratic surrogate and apply data processing to that surrogate.

For $s>0$, define
\begin{equation}
  \psi_s(u,v)
  :=
  \frac{|u-v|^2}{s^2+|u-v|^2}, u,v \in \R.
  \label{eq:smooth-cap}
\end{equation}
For all $u,v\in\R$,
\begin{equation}
  \frac12
  \left(1\wedge\frac{|u-v|^2}{s^2}\right)
  \le
  \psi_s(u,v)
  \le
  1\wedge\frac{|u-v|^2}{s^2}.
  \label{eq:smooth-hard-comparison}
\end{equation}
Moreover, for any $s>0$, the Cauchy kernel
\begin{equation}
  k_s(u,v):=1-\psi_s(u,v)=\frac{s^2}{s^2+(u-v)^2}
  \label{eq:cauchy-positive-kernel}
\end{equation}
is positive definite.  Indeed, it can be written as the Fourier transform of a Laplace distribution, so Bochner's theorem~\cite{bochner1959lectures} implies the result:
\begin{equation}
  k_s(u,v)
  =
  \int_{\R}e^{i\xi(u-v)}\,\nu_s(d\xi),
  \qquad
  \nu_s(d\xi):=\frac{s}{2}e^{-s|\xi|}\,d\xi.
  \label{eq:cauchy-kernel-fourier}
\end{equation}
Thus, for any $s>0$, the Moore--Aronszajn
theorem~\cite{Aronszajn1950} yields a real Hilbert-space feature map
$\Gamma_s$ satisfying, for all $u,v\in\R$,
\[
  \norm{\Gamma_s(u)}=1,
  \qquad
  \ip{\Gamma_s(u)}{\Gamma_s(v)}=k_s(u,v),
\]
and
\begin{equation}
  \psi_s(u,v)
  =
  \frac12\norm{\Gamma_s(u)-\Gamma_s(v)}^2.
  \label{eq:smooth-cap-hilbert}
\end{equation}

The quadratic structure gives the following useful property.  For any
channel with signal $U$, if $U^{(1)},U^{(2)}$ are i.i.d. posterior
samples of $U$ given an observation vector $\mathcal O$, then the
Nishimori identity from \cref{sec:infthbackground} gives
\begin{align}
  \E\bigl[\psi_s(U^{(1)},U^{(2)})\bigr]
  &=
 \frac{1}{2} \E\norm{
    \Gamma_s(U^{(1)})- \Gamma_s(U^{(2)})
  }^2=
  \E\norm{
    \Gamma_s(U)-\E[\Gamma_s(U)\mid\mathcal O]
  }^2.
  \label{eq:posterior-feature-variance}
\end{align}
Therefore
$\mathbb E[\psi_s(U^{(1)},U^{(2)})]$ can only decrease when the
observation becomes more informative.  Indeed, suppose
$U\to\mathcal O'\to\mathcal O$ is a Markov chain, and set
\[
  M':=\mathbb E[\Gamma_s(U)\mid\mathcal O'],
  \qquad
  M:=\mathbb E[\Gamma_s(U)\mid\mathcal O].
\]
The Markov property gives
$\mathbb E[\Gamma_s(U)\mid\mathcal O',\mathcal O]=M'$, and hence
\[
  \mathbb E\langle\Gamma_s(U)-M',M'-M\rangle=0.
\]
Since
$\Gamma_s(U)-M=(\Gamma_s(U)-M')+(M'-M)$, the Pythagorean identity
gives
\[
  \mathbb E\|\Gamma_s(U)-M\|^2
  =
  \mathbb E\|\Gamma_s(U)-M'\|^2
  +\mathbb E\|M'-M\|^2.
\]
Thus the posterior feature variance for $\mathcal O'$ is no larger
than that for $\mathcal O$; see Blackwell~\cite{Blackwell1953}.

Now specialize to the diagonal scale $\bm t=(t,\ldots,t)$ and take
$s\ge t$.  From the mixed observation
\[
  Y_{(t,\ldots,t,s,t,\ldots,t)}
\]
we can generate the all-$s$ observation $Y_s=X+sZ$ by adding, to
every coordinate $j\ne i$, an independent Cauchy noise of scale
$s-t$.  Indeed, for independent standard Cauchy variables $\xi,\xi'$
and $0\le t\le s$, stability gives
\begin{equation}
  t\xi+(s-t)\xi'\ \stackrel d=\ s\xi.
  \label{eq:cauchy-stability}
\end{equation}
Hence the mixed observation is more informative than the all-$s$
observation.  By \cref{eq:smooth-hard-comparison,eq:posterior-feature-variance}, we can then conclude
\begin{equation}
  \mathrm{rcMMSE}_{i,\mu}((t,\ldots,t)^{i,s})
  \le
  2\mathrm{rcMMSE}_{i,\mu}(s).
\end{equation}

\end{proof}

\begin{proof}[Proof of \cref{prop:multii_I-MMSE}]

Notice that using
\cref{eq:mixed-partial-bound,eq:mixed-to-diagonal-risk} we have
\begin{align}
  -I_\mu'(t)
  &=-\sum_{i=1}^n
  \partial_{t_i}I_\mu(t,\ldots,t)
  \le
  4\sum_{i=1}^n
  \int_t^\infty\frac{\mathrm{rcMMSE}_{i,\mu}(s)}{s^2}\,ds =
  4\int_t^\infty\frac{\mathrm{rcMMSE}_\mu(s)}{s^2}\,ds.
  \label{eq:product-derivative-tail}
\end{align}

Fix $x_0\in T$.  The information-radius identity, followed by
tensorization of relative entropy for the product Cauchy channel,
gives
\begin{align*}
  I_\mu(t)
  &\le
  \mathbb E_X
  D_{\mathrm{KL}}\bigl(
    P_{X+tZ\mid X}\,\|\,P_{x_0+tZ}
  \bigr) \\
  &=
  \sum_{i=1}^n
  \mathbb E_X
  D_{\mathrm{KL}}\bigl(
    P_{X_i+tZ_i\mid X_i}\,\|\,P_{x_{0,i}+tZ_i}
  \bigr) \\
  &\le
  \frac{\mathbb E\|X-x_0\|_1}{t}
  \longrightarrow0.
\end{align*}
Here the last inequality is the shifted-Cauchy estimate used in the
proof of \cref{lem:scalar-curvature}.  Thus
\cref{eq:product-derivative-tail} and Tonelli's theorem give
\begin{align}
  I_\mu(t)
  &=
  \int_t^\infty[-I_\mu'(u)]\,du\le
  4\int_t^\infty
  \int_u^\infty\frac{\mathrm{rcMMSE}_\mu(s)}{s^2}\,ds\,du\\
 & \qquad \qquad =
  4\int_t^\infty
  \frac{s-t}{s^2}\mathrm{rcMMSE}_\mu(s)\,ds
  \le
4\int_t^\infty\frac{\mathrm{rcMMSE}_\mu(s)}{s}\,ds.
  \label{eq:mutual-information-hardy-tail}
\end{align}
Using \cref{lem:replica-center_1} concludes the proof.
\end{proof}

\enlargethispage{2\baselineskip}
\bibliographystyle{alpha}
\bibliography{main}

\end{document}